\documentclass[11pt]{revtex4-2}
\usepackage[a4paper,margin=1in]{geometry}
\usepackage{amsmath,amssymb,amsthm,mathtools}
\usepackage{enumitem}
\usepackage[hidelinks]{hyperref}
\usepackage{booktabs}
\usepackage{microtype}

\newtheorem{theorem}{Theorem}[section]
\newtheorem{lemma}[theorem]{Lemma}
\newtheorem{proposition}[theorem]{Proposition}
\newtheorem{corollary}[theorem]{Corollary}
\newtheorem{definition}[theorem]{Definition}

\newcommand{\eps}{\varepsilon}
\newcommand{\Om}{\Omega}
\newcommand{\R}{\mathbb R}
\newcommand{\C}{\mathbb C}
\newcommand{\Z}{\mathbb Z}
\newcommand{\Sone}{\mathbb S^1}
\newcommand{\dd}{\,d}
\newcommand{\e}{\mathrm e}
\newcommand{\loc}{\mathrm{loc}}

\newcommand{\divop}{\operatorname{div}}

\newcommand{\ii}{i}
\begin{document}
\preprint{APS/123-QED}

\title{The Asymptotic Theory of Ginzburg--Landau Critical Points on Domains with Corners}

\author{Yucen Han}
 \affiliation{Center for Applied Mathematics, Renmin University of China, Beijing 100872, China}

\author{Apala Majumdar}
\affiliation{Department of Mathematics, University of Manchester, Manchester M13 9PL, United Kingdom}


\begin{abstract}
We study the critical points $U_\eps$ of the two-dimensional
Ginzburg--Landau energy on 
domains with corners $\Omega$ and imposed Dirichlet boundary condition which vanishes precisely at the vertices $V$. 
We isolate the forced scalar boundary layer $\rho_\eps$ by the weighted
Lassoued--Mironescu splitting.
Assuming the reduced logarithmic energy bound and vertex logarithmic tightness,
we prove that after passing to a subsequence, the normalized maps $v_\eps = U_\eps/\rho_\eps$
converges in $C^k_{\loc}(\Omega\setminus A)$
$\forall k\ge0$ away from a finite set of interior vortices $A$. The limiting map is a canonical $S^1$-valued
harmonic map $u_*$ containing integer interior vortex factors and fixed
fractional corner factors.
The renormalized energy is defined by removing discs around the
interior vortices and sectors around the vertices and subtracting the
corresponding logarithmic divergences. The interior vortex locations of the limiting map 
are critical points of this renormalized energy. 
\end{abstract}

\maketitle
\section{Introduction}

The asymptotic theory of two-dimensional Ginzburg--Landau vortices
was developed systematically by Bethuel, Brezis, and H\'elein (BBH)
\cite{BBH}. For smooth bounded domains with nonvanishing
$S^1$-valued Dirichlet data, they showed that solutions in the
logarithmic-energy regime converge, away from finitely many
singularities, to a canonical $S^1$-valued harmonic map. The
singularities carry integer degrees and interact through the
renormalized energy. For minimizers, the limiting vortex
configuration minimizes this energy, whereas for general
 critical points it is only required to be a critical
point; see \cite[Chapter~X]{BBH}.
Andr\'e and Shafrir
\cite{ASzeros} studied minimizing Ginzburg--Landau-type
problems on smooth domains with boundary zeros. Their work separates
the scalar boundary layer generated by the vanishing modulus from
the normalized phase and analyzes how minimizing vortices interact
with the boundary zeros.

The purpose of the present manuscript is to develop the corresponding
critical-point theory on domains with corners. 
Thus two difficulties, absent from the
classical BBH setting, occur simultaneously.  First, the domain has sharp
edges.  Near a vertex \(b_j\), the limiting singularity is not an interior
integer vortex but a fixed geometric corner defect determined by the
sector angle \(\omega_j\).  Second, the boundary condition has zeros, i.e., $\rho_b^{-1}(0)=V$ in \eqref{eq:rho_b_0},
and the boundary phase may jump with certain lift  at each
vertex,  i.e., $\eta_j=\varphi_j^+-\varphi_j^-+2\pi m_j=2\pi c_j$ in \eqref{eq:jump} and \eqref{eq:lifted-jump}. 

We overcome the lack of smooth boundary and the degeneracy at the vertices in the following three steps. First, we introduce
the scalar minimizer $\rho_\varepsilon$ with boundary value
$\rho_b$ and normalize the critical point by $
v_\varepsilon
=
\frac{U_\varepsilon}{\rho_\varepsilon}.
$
The weighted Lassoued--Mironescu splitting is used to separates the energy of
the forced scalar boundary layer from the reduced phase and vortex energy. 
Second, we remove fixed sectors around the vertices and work on the
truncated domains $\Omega_\delta$. On every such fixed truncation,
$\rho_\varepsilon$ is uniformly positive and the normalized equation
is uniformly elliptic. A weighted local version of the improved
potential estimate of Bethuel--Orlandi--Smets \cite{BOS} gives a uniform
bound on the number of bad discs, after which the local
critical compactness argument of BBH can be applied in
interior discs and flattened boundary half-discs.
Third, we let $\delta$ tend to zero and use a diagonal exhaustion of
$\Omega\setminus V$. Bad cores that stay away from the vertices give
a finite set of interior vortices, while all unresolved cores
approaching a vertex are absorbed into shrinking vertex sectors.
This procedure yields compactness on every compact subset away from
the interior vortex set and the vertices. 

Our main results are summarized as follows. For every reduced
admissible family of  critical points defined in Definition \ref{def:reduced-admissible-main}, there exist a
subsequence $\eps_\ell$, a finite set of interior points
\[
A=\{a_1,\ldots,a_N\}\Subset\Omega,
\]
and nonzero integers $d_1,\ldots,d_N$ such that
\[
v_{\varepsilon_\ell}
\rightarrow u_*
\quad\text{in }
C^k_{\mathrm{loc}}
\bigl(\Omega\setminus(A\cup V)\bigr)
\quad \forall k\geq0.
\]
The limit $u_*$ is a canonical $S^1$-valued harmonic map. The limiting map is an analogue of the BBH canonical harmonic map. At a vertex $b_j$, $j = 1,\cdots, K$,
however, the singularity is not a free vortex location.  It is forced by the geometry
and by the discontinuity of the boundary phase measured by $c_j$, the signed corner winding.
The renormalized energy $W$ is then defined by the same shrinking-hole philosophy as
in BBH, but the holes now include both small discs around the free interior vortices
and small sectors around the fixed vertices.  
A local stress--energy
argument around each interior vortex gives
\[
\nabla_{a_i}W=0,
\qquad i=1,\ldots,N.
\]
Thus the limiting interior vortex configuration is a critical point
of the polygonal renormalized energy.
This agrees with BBH Chapter X and with Lin's \cite{lin1995solutions} viewpoint that Ginzburg--Landau
critical points are related to critical points of the renormalized energy, rather
than necessarily to minimizers of it. 

\section{Preliminaries and Setup}
\label{sec:preliminaries}
\begin{figure}
    \centering
    \includegraphics[width=0.4\linewidth]{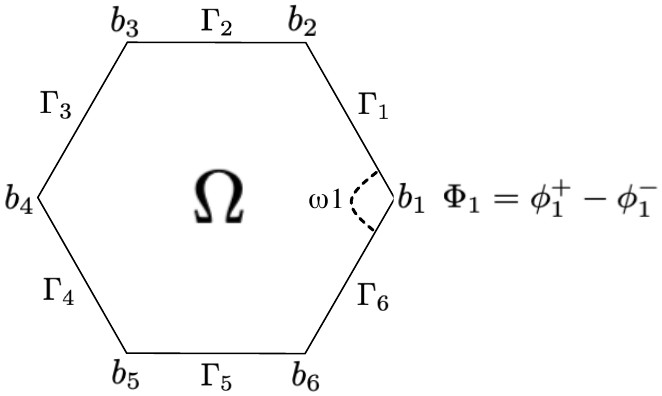}
    \caption{Polygonal geometry and phase jump at vertices.  The boundary is a union of
smooth open sides $\Gamma_k$ meeting at vertices \(b_j\) with interior angle $\omega_j$.  The modulus
\(\rho_b\) vanishes only at \(\{b_1,\cdots,b_K\}\), and the phase has one-sided limits
\(\varphi_j^\pm\), producing the jump \(\Phi_j\).}
    \label{fig:Omega}
\end{figure}
The Ginzburg-Landau free energy is 
\begin{equation}\label{eq:GL_energy}
  E_\eps(U)=\frac12\int_\Om |\nabla U|^2\,dx+
  \frac1{4\eps^2}\int_\Om (1-|U|^2)^2\,dx.
\end{equation}

Let $\Om\subset\R^2\simeq\C$ be a bounded simply connected polygonal or curvilinear polygonal domain. Its boundary consists of finitely many open $C^2$ arcs $\Gamma_1,\ldots,\Gamma_K$,
meeting at vertices
$V=\{b_1,\ldots,b_K\}$.
The interior angle at $b_j$ is denoted by $\omega_j\in(0,2\pi)$. For a straight polygon, a neighbourhood of $b_j$ is exactly the sector
\[
  S_{j}(R)=\{b_j+re^{i\theta}:0<r<R,\ 0<\theta<\omega_j\},
\]
after rotation and translation.
For a curvilinear polygon there is a $C^2$ diffeomorphism $\Psi_j:S_{j}(R)\to\Om\cap U_j$, where $U_j$ is the open neighbourhoods of the vertices $b_j$,
\[
  \Psi_j(y)=b_j+Q_j(y-b_j)+O(|y-b_j|^2),
\]
where $Q_j$ is a rotation and dilation. Hence
\[\Psi_j(b_j)=b_j,
D\Psi_j(b_j)=Q_j,\qquad
  (D\Psi_j)^TD\Psi_j=Q_j^TQ_j+O(r),\qquad r=|y-b_j|.
\]

The Dirichlet boundary condition is
\begin{equation}\label{eq:gdatum}
g=\rho_b e^{i\varphi_b}
\quad\text{on }\partial\Omega\setminus V.
\end{equation}
The modulus $\rho_b:\partial\Omega\to[0,1]$ satisfies
\begin{equation}\label{eq:rho_b_0}
\rho_b^{-1}(0)=V,
\end{equation}
and $\rho_b$ is $C^2$ on each open side $\Gamma_j$ away from the endpoints $b_j$. In particular, we assume the modulus has algebraic zero of order $\alpha_j>0$
\begin{equation}\label{eq:modulus-zero}
  \rho_b(x)=|x-b_j|^{\alpha_j}h_j^\pm(x),
  \qquad h_j^\pm(b_j)>0,
\end{equation}
along the two sides adjacent to $b_j$, where $h_j^{\pm}$ are $C^2$ up to $b_j$ on the two adjacent sides.
Since near each vertex the one-dimensional
$H^{1/2}$ seminorm satisfies
\[
\int_0^\delta\int_0^\delta
\frac{|r^{\alpha_j}-s^{\alpha_j}|^2}{|r-s|^2}\,drds<C,
\]
we have \(\rho_b\in H^{1/2}(\partial\Omega)\).

On each open side $\Gamma_j$ the phase
$\varphi_b\in C^2(\overline{\Gamma}_j;\mathbb{R})$ satisfies
\[
g_{ph}:=e^{i\varphi_b}=\frac{g}{|g|}
\quad\text{on }\Gamma_j.
\]
The boundary current is
\begin{equation}\label{eq:boundary_current}
  q:=g_{\rm ph}\times\,\partial_\tau g_{\rm ph}
  =\partial_\tau\varphi_b
  \quad\text{on }\partial \Omega\setminus V,
\end{equation}
where $\tau$ is a unit tangential vector to $\partial \Omega\setminus V$.
At $b_j$ the one-sided phase limits are
\[
  \varphi_j^-:=\lim_{x\to b_j,\ x\in\Gamma_{j-1}}\varphi_b(x),
  \qquad
  \varphi_j^+:=\lim_{x\to b_j,\ x\in\Gamma_j}\varphi_b(x),
\]
and the prescribed jump is
\begin{equation}\label{eq:jump}
  \Phi_j:=\varphi_j^+-\varphi_j^-.
\end{equation}
A choice of lift supplies $m_j\in\Z$ and a lifted jump
\begin{equation}
  \eta_j=\Phi_j+2\pi m_j=2\pi c_j.
  \label{eq:lifted-jump}
\end{equation}
Thus $c_j$ is the signed corner winding.  

The solutions of 
the corresponding Euler-Lagrange equation is
\begin{equation}\label{eq:GL-EL}
  -\Delta U_\eps=\eps^{-2}U_\eps(1-|U_\eps|^2)
  \quad\text{in }\Om.
\end{equation}
The critical points of the Ginzburg-Landau free energy in \eqref{eq:GL_energy} are the solution of \eqref{eq:GL-EL} in $H^1(\Omega)$.
\section{Main Theorem}

\begin{theorem}[Critical polygonal Ginzburg--Landau limit]
\label{thm:polygon-limit}
Let $\Om\subset\R^2\simeq\C$ be a bounded simply connected polygonal or curvilinear polygonal domain. Its boundary consists of finitely many open $C^2$ arcs $\Gamma_1,\ldots,\Gamma_K$,
meeting at vertices
\[V=\{b_1,\ldots,b_K\}.\]
The interior angle at $b_j$ is denoted by $\omega_j\in(0,2\pi)$. 
The Dirichlet boundary condition is 
$g=\rho_b e^{i\varphi_b}$ on $\partial\Omega\setminus V$.
The modulus $\rho_b:\partial\Omega\to[0,1]$ satisfies
$
\rho_b^{-1}(0)=V$,
$\rho_b$ is $C^2$ on each open side $\Gamma_j$ away from the endpoints $b_j$, and \(\rho_b\in H^{1/2}(\partial\Omega)\).
On each open side $\Gamma_j$ the phase
$\varphi_b\in C^2(\overline{\Gamma}_j;\mathbb{R})$ satisfies
$g_{ph}:=e^{i\varphi_b}=\frac{g}{|g|}$ on $\Gamma_j$.
At $b_j$, the lifted phase jump is
\begin{equation}
  \eta_j=\lim_{x\to b_j,\ x\in\Gamma_j}\varphi_b(x)-\lim_{x\to b_j,\ x\in\Gamma_{j-1}}\varphi_b(x)+2\pi m_j=2\pi c_j.
\end{equation}
where $m_j\in\Z$ and $c_j$ is the signed corner winding. 

Let
$U_\eps\in H^1(\Omega;\C)$ be a reduced
admissible family of  critical points of
\eqref{eq:GL_energy} defined in Definition \ref{def:reduced-admissible-main} with trace $g$.  Let $\rho_\eps$ be the scalar solution of \eqref{eq:scalar-profile} and
set $v_\eps=U_\eps/\rho_\eps$. Then there are a sequence
$\eps_\ell\rightarrow0$, an integer $N\ge0$, distinct interior points
\[
  A=\{a_1,\ldots,a_N\}\Subset \Omega,
\]
 nonzero winding number $d_1,\ldots,d_N\in\Z$, 
for which the configuration $\mathcal C=(a_i,d_i;b_j,c_j)$  is \emph{topologically compatible} with the normalized
boundary phase
\begin{equation}
  \sum_{i=1}^N d_i
  =\frac1{2\pi}\int_{\partial \Omega\setminus V}q\,\dd s
  +\sum_{j=1}^Kc_j,
  \label{eq:topological-balance}
\end{equation}
such that the normalized maps converge locally
smoothly away from the interior vortex set and the vertices,
\[
  v_{\eps_\ell}\rightarrow u_*
  \quad\text{in }C^k_{\rm loc}
  \bigl(\Omega\setminus(A\cup V)\bigr)
  \quad \forall k\ge0.
\]
The limit $u_*$ is the unique proper canonical $S^1$-valued harmonic map
$u_*=u_{\mathcal C}$ defined by
\eqref{eq:canonical-map-main}--\eqref{eq:canonical-boundary-lift}.
The associated renormalized energy $W$ is given by \eqref{eq:general-W} and every interior vortex
location $a_i$ satisfies
\begin{equation}
  \nabla_{a_i}W(\mathcal C)=0,
  \qquad i=1,\ldots,N.
  \label{eq:main-criticality}
\end{equation}
\end{theorem}
\begin{proof}
Proposition~\ref{prop:scalar-reduction} uses the  weighted Lassoued–Mironescu splitting to get the normalized map $v_{\eps}$ and Definition~\ref{def:reduced-admissible-main} supplies the
logarithmic energy scale.  Proposition~\ref{prop:fixed-truncation-main}
then yields the uniform local potential estimate required for the bad-disc
construction.  Theorem~\ref{thm:clearing-main} gives a finite interior
singular set, nonzero integer degrees, and local smooth convergence away
from that set and the vertices.

The vertex-tightness part of Definition~\ref{def:reduced-admissible-main}
and Lemma~\ref{pf:lem:vertex-properness} exclude additional negative
sector modes.  Proposition~\ref{prop:canonical-main} then supplies the
lifted corner jumps, proves the topological balance
\eqref{eq:topological-balance}, and identifies the limit with the unique canonical map
\eqref{eq:canonical-map-main}.  Proposition~\ref{prop:finite-part-main}
constructs its renormalized energy and identifies it with
\eqref{eq:general-W}.  Finally,
Proposition~\ref{prop:derivative-main} and
Theorem~\ref{thm:vanishing-main} give the conditions
$\nabla_{a_i}W=0$ for all interior vortices.  These statements are exactly
the compactness, canonical-map, finite-part, and criticality conclusions
of Theorem~\ref{thm:polygon-limit}.
\end{proof}

\section{Proof of Main Theorem}
\subsection{Scalar profile $\rho_{\eps}$ and Weighted Lassoued--Mironescu splitting }

For each $\eps>0$, let $\rho_\eps$ be the non-negative minimizer of the scalar minimization problem, 
\begin{equation}\label{eq:rho-min}
  \rho_\eps=\operatorname*{argmin}\left\{
    E_{\eps}(\rho):
    \rho\in H_{\rho_b}^1(\Om),\rho\geq0
  \right\},
\end{equation}
which satisfies
\begin{equation}\label{eq:rho-pde}
  -\Delta\rho_\eps=\eps^{-2}\rho_\eps(1-\rho_\eps^2).
\end{equation}
$E_\eps$ is defined as
\begin{equation}
  E_\eps(\rho)=\frac12\int_\Om |\nabla \rho|^2\,dx+
  \frac1{4\eps^2}\int_\Om (1-\rho^2)^2\,dx.
\end{equation}

\begin{lemma}[Scalar estimates used below]\label{pf:lem:scalar}
The unique scalar minimizer $\rho_\eps$ satisfies the following estimates.
\begin{enumerate}[label=(\roman*)]
\item $0\le \rho_\eps\le1$ in $\Omega$.
\item If \(\mathcal K\Subset\overline \Omega\setminus V\), and if
\(\mathcal K\cap\partial \Omega\) is contained in compact subarcs of the open sides
\(\Gamma_i\), then there exists \(c_{\mathcal K}>0\) such that
\[
\rho_\eps\ge c_{\mathcal K}
\quad\text{on }\mathcal{K}\cap \Omega,
\]
for all sufficiently small $\eps$.
\item If $\mathcal K\Subset \Omega$ is an interior compact set, then
\[
  |1-\rho_\eps|+\eps|\nabla\rho_\eps|\le C_{\mathcal K} e^{-c_{\mathcal K}/\eps}
  \quad\text{on }{\mathcal K}.
\]
\end{enumerate}
\end{lemma}

\begin{proof}
Existence follows from the direct method and the weak maximum principle and the uniqueness follows from the Brezis--Oswald theorem for scalar semilinear equations \cite{BO}.

For (i), $0\le\rho_\eps\le1$ follow from the direct method and the weak maximum principle.

For (ii), let \(h\) be the harmonic extension of \(\rho_b\):
\[
  \begin{cases}
  -\Delta h=0 &\text{in }\Omega,\\
  h=\rho_b &\text{on }\partial \Omega.
  \end{cases}
\]
Since \(\rho_b\ge0\), \(\rho_b\not\equiv0\), and
\(\rho_b^{-1}(0)=V\), the strong maximum principle gives
\(h>0\) in \(\Omega\).  Moreover, if
\(\mathcal K\Subset\overline \Omega\setminus V\) and
\(\mathcal K\cap\partial \Omega\) is contained in compact subarcs of the open sides,
then \(\rho_b\ge m_{\mathcal K}>0\) on those subarcs.  By continuity of the harmonic
extension up to the \(C^2\) side pieces, and by compactness of the interior
part of \(\mathcal K\), there exists \(c_{\mathcal K}>0\) such that
\[
  h\ge c_{\mathcal K}\quad\text{on }\mathcal K\cap \Omega .
\]

We now compare \(\rho_\eps\) with \(h\).  Since \(0\le\rho_\eps\le1\),
the scalar equation gives
\[
  -\Delta\rho_\eps
  =
  \eps^{-2}\rho_\eps(1-\rho_\eps^2)\ge0.
\]
Thus \(\rho_\eps\) is superharmonic.  Set
\[
  \psi=\max\{h-\rho_\eps,0\}\in H^1_0(\Omega).
\]
The weak formulation of the scalar equation gives
\[
  \int_\Omega \nabla\rho_\eps\cdot\nabla\psi
  =
  \eps^{-2}\int_\Omega
  \rho_\eps(1-\rho_\eps^2)\psi
  \ge0,
\]
while harmonicity of \(h\) gives
\[
  \int_\Omega \nabla h\cdot\nabla\psi=0.
\]
Therefore
\[
  \int_\Omega \nabla(\rho_\eps-h)\cdot\nabla\psi\ge0.
\]
On the set \(\{h>\rho_\eps\}\), one has
\(\psi=h-\rho_\eps\), hence
\[
  \nabla(\rho_\eps-h)=-\nabla\psi.
\]
Consequently
\[
  -\int_\Omega|\nabla\psi|^2\ge0,
\]
so \(\psi\equiv0\).  Hence
\[
  \rho_\eps\ge h\ge c_{\mathcal K}
  \quad\text{on }\mathcal K\cap \Omega .
\]

For (iii), the exponential convergence away from the boundary-zero neighbourhoods in (iii) is the standard scalar GL boundary-layer estimate proved by sub/supersolutions and local elliptic gradient estimates. The proof is analogous to Proposition 2.1 in \cite{ASnonvanishing}.
\end{proof}

\begin{proposition}[Scalar reduction and weighted splitting]
\label{prop:scalar-reduction}
The scalar minimizer is unique and satisfies
\begin{equation}
  -\Delta\rho_\eps
  =\eps^{-2}\rho_\eps(1-\rho_\eps^2),
  \qquad 0<\rho_\eps\le1\quad\text{in }\Omega.
  \label{eq:scalar-profile}
\end{equation}
Every solution of \eqref{eq:GL-EL}, $U_\eps$, satisfies $|U_\eps|\le\rho_\eps$.  Consequently,
\[
  v_\eps:=\frac{U_\eps}{\rho_\eps}
\]
is well defined in $\Omega$, satisfying $|v_\eps|\le1$, and has trace
$g_{\rm ph}$ on the physical boundary away from $V$.  Moreover,
\begin{equation}
  E_\eps(U_\eps)=E_\eps(\rho_\eps)+F_\eps(v_\eps),
  \label{eq:LM-main}
\end{equation}
where
\begin{equation}
  F_\eps(v)
  =\frac12\int_\Omega\rho_\eps^2|\nabla v|^2\,\dd x
  +\frac1{4\eps^2}\int_\Omega
  \rho_\eps^4(1-|v|^2)^2\,\dd x.
  \label{eq:reduced-energy-main}
\end{equation}
The normalized map $v_{\eps}$ is a critical point of $F_\eps$ and solves
\begin{align}
  -\divop(\rho_\eps^2\nabla v_\eps)
  &=\eps^{-2}\rho_\eps^4v_\eps(1-|v_\eps|^2)\ on\ \Omega,
  \label{eq:reduced-equation-main}\\
  v_{\eps} &= g_{ph}\ on\ \partial\Omega\backslash V.
\end{align}
\end{proposition}
\begin{proof}
Existence, uniqueness, the scalar equation, and strict interior positivity
follow from Lemma~\ref{pf:lem:scalar}.  Following the standard
Kato comparison \cite{Kato1972}, Kato's inequality applied to
\eqref{eq:GL-EL}, followed by the scalar comparison argument, gives
$|U_\eps|\leq\rho_\eps$ almost everywhere in $\Omega$.  Since the proof of
Lemma~\ref{pf:lem:scalar}(ii) gives $\rho_\eps\geq h>0$ in $\Omega$, the quotient
$v_\eps=U_\eps/\rho_\eps$ is well defined and satisfies $|v_\eps|\leq1$.
On the physical boundary away from $V$ its trace is $g_{\rm ph}$.  No quotient trace is
prescribed at a vertex, where $\rho_b=0$.

The equations in \eqref{eq:LM-main}--\eqref{eq:reduced-energy-main} are obtained by substituting $U_{\eps} = \rho_{\eps}v_{\eps}$ into $E_{\eps}(U_{\eps})$.  On the open sides the
boundary term vanishes because $|v|=1$; at the vertices one first removes
small sectors and then lets their radii tend to zero.  This is the weighted
Lassoued--Mironescu splitting \cite{LM,ASzeros}.  Taking its first variation
at $v_\eps$ gives \eqref{eq:reduced-equation-main}, completing the proof.
\end{proof}

\begin{definition}[Reduced admissibility]
\label{def:reduced-admissible-main}
A family $U_\eps$ is reduced admissible when $v_\eps$ solves
\eqref{eq:reduced-equation-main} and
\begin{equation}
  F_\eps(v_\eps)\le C\ln\frac1\eps
  \label{eq:reduced-log-bound-main}
\end{equation}
with $C$ independent of $\eps$.  We additionally require
vertex logarithmic tightness: for every $j$ there are $r_0>0$ and $C_j>0$
such that
\begin{equation}
  \int_{S_j(r_0)\setminus S_j(r)}
  \rho_\eps^2|\nabla v_\eps|^2\,\dd x
  \leq C_j\left(1+\ln\frac{r_0}{r}\right),
  \qquad \eps\leq r<r_0.
  \label{eq:vertex-log-tightness}
\end{equation}
\end{definition}

The global
bound \eqref{eq:reduced-log-bound-main} alone controls the total number of
interior cores but does not rule out extra zero-trace negative sector modes
concentrating at a vertex.  Condition \eqref{eq:vertex-log-tightness} is
used to exclude those modes and hence to identify the limit as a
proper canonical map; see Lemma~\ref{pf:lem:vertex-properness}.

\begin{lemma}\label{pf:lem:gradscale}
Let $v_\eps$ solve \eqref{eq:reduced-equation-main}. It satisfies $|v_\eps|\le1$. If $\mathcal K\Subset \overline \Omega\setminus V$ and $\mathcal K$ meets the boundary only on smooth side pieces, then
\[
  \|\nabla v_\eps\|_{L^\infty(\mathcal K\cap \Omega)}\le \frac{C_{\mathcal K}}{\eps}.
\]
\end{lemma}

\begin{proof}
The bound $|v_\eps|\le1$ is Proposition~\ref{prop:scalar-reduction}. Away from the vertex set $V$, the coefficient $\rho_\eps^2$ is uniformly elliptic according to Lemma~\ref{pf:lem:scalar}(ii). Thus \eqref{eq:reduced-equation-main} is a uniformly elliptic GL-type system on $\mathcal K$, with right-hand side bounded by $C_{\mathcal K}\eps^{-2}$. The standard interior or flattened-boundary $C^{1,\alpha}$ estimate after rescaling balls of radius $\eps$ gives $\|\nabla v_\eps\|_{L^\infty(\mathcal K\cap \Omega)}\le C_{\mathcal K}/\eps$ \cite{Evans}.
\end{proof}
\subsection{Fixed truncations and the potential estimate}
\begin{figure}
    \centering
    \includegraphics[width=0.35\linewidth]{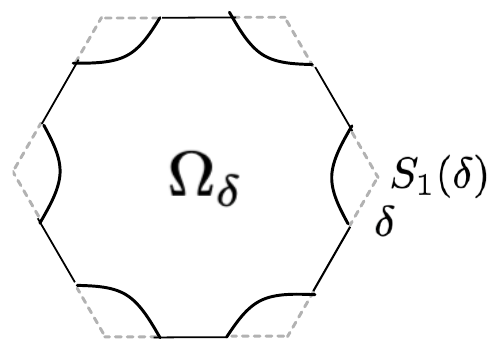}
    \caption{Fixed truncations by removing the vertex sectors
\(S_j(\delta)\) gives \(\Omega_\delta\).}
    \label{fig:Omega_delta}
\end{figure}
For fixed $\delta>0$, remove mutually disjoint vertex sectors and set
\begin{equation}
  \Omega_\delta
  :=\Omega\setminus\bigcup_{j=1}^KS_j(\delta).
  \label{eq:fixed-truncation-main}
\end{equation}
The quotient equation in \eqref{eq:reduced-equation-main} is degenerate only at the vertices.  Removing fixed
sectors restores the uniform ellipticity needed by the local BBH
compactness argument.


\begin{lemma}[Scalar profile on fixed truncations]\label{pf:lem:scalar-fixed-truncation}
For every fixed \(\delta>0\) there are constants
\(c_\delta,C_\delta>0\) such that, for all sufficiently small \(\eps\),
\begin{equation}\label{pf:eq:rho-fixed-lower}
  c_\delta\le \rho_\eps\le1
  \quad\text{on }\Omega_\delta,
\end{equation}
and
\begin{equation}
\label{pf:eq:rho-fixed-gradient}
\eps
\|\nabla\rho_\eps\|_{L^\infty(\Omega_\delta)}
\leq C_\delta.
\end{equation}
Consequently the coefficients
\[
  a_\eps:=\rho_\eps^2,
  \qquad
  b_\eps:=\rho_\eps^4
\]
satisfy, 
\begin{equation}
\label{pf:eq:coeff-fixed-ellipticity}
0<c_\delta^2\leq a_\eps\leq 1,
\qquad
0<c_\delta^4\leq b_\eps\leq 1
\qquad\text{in }\Omega_\delta,
\end{equation}
and
\begin{equation}
\label{pf:eq:coeff-fixed-gradient}
\eps
\left(
\|\nabla a_\eps\|_{L^\infty(\Omega_\delta)}
+
\|\nabla b_\eps\|_{L^\infty(\Omega_\delta)}
\right)
\leq C_\delta.
\end{equation}
\end{lemma}

\begin{proof}
The lower bound \eqref{pf:eq:rho-fixed-lower} follows from
Lemma~\ref{pf:lem:scalar}(ii).

We prove the derivative estimates.  First let
\(B_{2r}(x_0)\Subset \Omega\setminus V\).  For
\(x\in B_r(x_0)\) and for \(\eps\) small enough so that
\(x+\eps B_2\subset B_{2r}(x_0)\), define
\[
  \widetilde\rho(y)=\rho_\eps(x+\eps y),
  \qquad y\in B_2 .
\]
Substituting the above equation into \eqref{eq:scalar-profile}, then
\[
  -\Delta_y\widetilde\rho
  =
  \widetilde\rho(1-\widetilde\rho^2)
  \quad\text{in }B_2,
\]
and, by Lemma~\ref{pf:lem:scalar}(i),
\[
  0\le \widetilde\rho\le1 .
\]
Interior \(W^{2,p}\) estimates applied on \(B_2\), followed by the
Sobolev embedding \(W^{2,p}(B_1)\hookrightarrow C^{1,\alpha}(B_1)\)
for \(p>2\), give
\[
  \|\nabla_y\widetilde\rho\|_{L^\infty(B_1)}
  \le C.
\]
Scaling back gives
\[
  \eps|\nabla_x\rho_\eps(x)|\le C\qquad in\ B_{2r}(x_0).
\]

We now consider $U$ which satisfies $U\cap\partial \Omega\neq\emptyset$ and away from \(V\).  Let
\[
  \Psi:B^+_{2r}\to \Omega\cap U,
  \qquad
  B^+_{2r}:=\{y=(y_1,y_2): |y|<2r,\ y_2>0\},
\]
be a \(C^2\) flattening chart such that
\[
  \Psi(\{y_2=0\})=\partial \Omega\cap U.
\]

(A local chart means a local coordinate
representation of the domain by a standard disc or half-disc.
At an interior point, the chart is simply an ordinary ball
\[
B_{2r}(x_0)\Subset \Omega,
\]
and no change of coordinates is required. Near a smooth portion of
the physical boundary, a boundary chart is given by a $C^2$
diffeomorphism
\[
\Psi:B_{2r}^+\to \Omega\cap U
\]
such that
\[
\Psi(\{y_2=0\}\cap B_{2r})
=
\partial \Omega\cap U.
\]
Thus, in the boundary chart, the physical boundary is represented by
the flat part of a half-disc.)

Set
\[
  R_\eps(y)=\rho_\eps(\Psi(y)).
\]
Let
\[
  J(y)=|\det D\Psi(y)|,\qquad
  A(y)=J(y)D\Psi(y)^{-1}D\Psi(y)^{-T}.
\]
Substituting the above equations into \eqref{eq:scalar-profile}, \(R_\eps\) satisfies
\[
  -\operatorname{div}_y(A(y)\nabla_y R_\eps)
  =
  \eps^{-2}J(y)R_\eps(1-R_\eps^2)
  \quad\text{in }B^+_{2r},
\]
with boundary value
\[
  R_\eps(y_1,0)=\rho_b(\Psi(y_1,0)).
\]
Fix \(y_0=(s_0,0)\) on the flat boundary and rescale
\[
  \widetilde R_\eps(z)
  =
  R_\eps(y_0+\eps z),
  \qquad z\in B^+_2 .
\]
Then
\[
  -\operatorname{div}_z(A_\eps(z)\nabla_z\widetilde R_\eps)
  =
  J_\eps(z)\widetilde R_\eps(1-\widetilde R_\eps^2),
\]
where
\[
  A_\eps(z)=A(y_0+\eps z),
  \qquad
  J_\eps(z)=J(y_0+\eps z).
\]
The matrices \(A_\eps\) are uniformly elliptic and uniformly \(C^{0,1}\)
on \(B^+_2\), with constants depending only on the fixed side chart.
The boundary value is
\[
  \widetilde R_\eps(z_1,0)
  =
  \rho_b(\Psi(s_0+\eps z_1,0)).
\]
Since \(\rho_b\) is \(C^2\) on the side piece,
\[
  \|\widetilde R_\eps(\cdot,0)\|_{W^{2,p}(-2,2)}
  \le C_p .
\]
Boundary \(W^{2,p}\) estimates for uniformly elliptic divergence-form
equations with \(C^{0,1}\) coefficients give
\[
  \|\widetilde R_\eps\|_{W^{2,p}(B^+_1)}
  \le C_p .
\]
Taking \(p>2\) and using Sobolev embedding gives
\[
  \|\nabla_z\widetilde R_\eps\|_{L^\infty(B^+_1)}
  \le C.
\]
Scaling back yields
\[
  \eps|\nabla_x\rho_\eps|\le C\qquad in\ \Omega\cap U
\]

Thus the scaled gradient estimate holds on every smaller boundary
patch. Since the covering is finite, taking the maximum of the
finitely many local constants proves
\[
\eps
\|\nabla\rho_\eps\|_{L^\infty(\Omega_\delta)}
\leq C_\delta,
\]
which is \eqref{pf:eq:rho-fixed-gradient}.

Finally,
\[
\nabla a_\eps
=
2\rho_\eps\nabla\rho_\eps,
\qquad
\nabla b_\eps
=
4\rho_\eps^3\nabla\rho_\eps.
\]
Using $0\leq\rho_\eps\leq 1$ and
\eqref{pf:eq:rho-fixed-gradient}, we obtain
\[
\eps
\left(
\|\nabla a_\eps\|_{L^\infty(\Omega_\delta)}
+
\|\nabla b_\eps\|_{L^\infty(\Omega_\delta)}
\right)
\leq
6\eps
\|\nabla\rho_\eps\|_{L^\infty(\Omega_\delta)}
\leq C_\delta.
\]
The ellipticity bounds in \eqref{pf:eq:coeff-fixed-ellipticity} follow
directly from \eqref{pf:eq:rho-fixed-lower}.
\end{proof}

We use the following fixed-truncation version of the local estimate of Bethuel--Orlandi--Smets (BOS).  Their theorem is local and independent of boundary conditions for the standard GL equation; in dimension two it implies that a logarithmic energy bound gives an $O(1)$ local potential bound.  The proposition below records the uniformly elliptic, coefficient-perturbed form used in the proof of Proposition~\ref{prop:fixed-truncation-main}.
\begin{proposition}[Weighted BOS estimate on fixed uniformly elliptic charts]
\label{pf:prop:weighted-bos}
Let \(D\) be either the unit disc \(B_1\) or the upper half-disc
\(B^+_1\).  Let \(D'\Subset D\cup(\partial D\cap\{x_2=0\})\) be a
concentric smaller disc or half-disc.  Let \(w_\eps:D\to\mathbb C\)
satisfy \(|w_\eps|\le1\) and
\begin{equation}\label{pf:eq:weighted-bos-eqn}
  -\operatorname{div}(A_\eps\nabla w_\eps)
  =
  \eps^{-2}B_\eps w_\eps(1-|w_\eps|^2)
  +
  \mathcal R_\eps
  \quad\text{in }D,
\end{equation}
where \(A_\eps(x)\) is a real symmetric \(2\times2\) matrix and
\(B_\eps(x)\) is scalar.  Assume that
\begin{equation}\label{pf:eq:coeff-bos}
  \lambda_{\rm ell}|\xi|^2
  \le
  A_\eps(x)\xi\cdot\xi
  \le
  \Lambda_{\rm ell}|\xi|^2,
  \qquad
  \lambda_{\rm ell}\le B_\eps(x)\le\Lambda_{\rm ell},
\end{equation}
and
\begin{equation}\label{pf:eq:coeff-bos-derivative}
  \eps\|\nabla A_\eps\|_{L^\infty(D)}
  +
  \eps\|\nabla B_\eps\|_{L^\infty(D)}
  \le\Lambda_{\rm ell} .
\end{equation}
In the half-disc case assume that \(w_\eps=1\) on the flat part of the
boundary.  The perturbation \(\mathcal R_\eps\) is allowed to be any
first-order term produced by a \(C^2\) boundary gauge.  More precisely,
we assume that, for every test function \(\zeta\),
\begin{equation}\label{pf:eq:Rbound}
  |\langle\mathcal R_\eps,\zeta\rangle|
  \le
  \Lambda_{\rm ell}
  \int_D
  \bigl(
  |\nabla w_\eps|\,|\zeta|
  +
  |\zeta|
  +
  |\nabla\zeta|
  \bigr).
\end{equation}
If
\begin{equation}\label{pf:eq:bos-log-energy}
  \int_D |\nabla w_\eps|^2
  +
  \frac1{\eps^2}\int_D(1-|w_\eps|^2)^2
  \le
  M\ln\frac1\eps,
\end{equation}
then
\begin{equation}\label{pf:eq:bos-potential-conclusion}
  \frac1{\eps^2}
  \int_{D'}(1-|w_\eps|^2)^2
  \le C,
\end{equation}
where \(C\) depends only on \(D'\), \(M\), \(\lambda_{\rm ell}\), and \(\Lambda_{\rm ell}\).
\end{proposition}

\begin{proof}
When \(A_\eps=I\), \(B_\eps=1\), and \(\mathcal R_\eps=0\), this is
Theorem~1.1 of Bethuel--Orlandi--Smets \cite[Theorem~1.1]{BOS}.  In
dimension two, their estimate reads, on \(B_1\),
\[
  \int_{B_{1/2}}
  \frac{(1-|u_\eps|^2)^2}{4\eps^2}
  \le
  C_0
  \frac{E_\eps(u_\eps;B_1)}{|\ln\eps|}
  \ln\left(
  2+
  \frac{E_\eps(u_\eps;B_1)}{|\ln\eps|}
  \right).
\]
Thus a logarithmic energy bound implies an \(O(1)\) bound for the
local potential.  The same paper also records the corresponding
covering and scaling form \cite[Corollary~1.4]{BOS}.

We use this estimate in a uniformly elliptic perturbative form.  The
assumptions \eqref{pf:eq:coeff-bos} imply that the kinetic energy associated
with \(A_\eps\) is equivalent to the Euclidean kinetic energy:
\[
  \lambda_{\rm ell}\int_D|\nabla w_\eps|^2
  \le
  \int_D A_\eps\nabla w_\eps\cdot\nabla w_\eps
  \le
  \Lambda_{\rm ell}\int_D|\nabla w_\eps|^2.
\]
The coefficient derivatives enter the local estimates only after
rescaling to the GL scale.  Indeed, on a ball \(B_\eps(x_0)\), set
\(x=x_0+\eps y\).  Then
\[
  \nabla_y A_\eps(x_0+\eps y)
  =
  \eps\,\nabla_x A_\eps(x_0+\eps y),
\]
and similarly for \(B_\eps\).  Hence
\eqref{pf:eq:coeff-bos-derivative} gives uniformly bounded coefficients
in the rescaled equation.  Repeating the local Caccioppoli inequalities,
the modulus equation estimates, and the supercurrent estimates in the
proof of \cite[Theorem~1.1]{BOS}, all coefficient-error terms are bounded
by the same quantities as in the unweighted case, with constants depending
only on \(\lambda_{\rm ell}\) and \(\Lambda_{\rm ell}\).

The perturbation \(\mathcal R_\eps\) satisfies
\eqref{pf:eq:Rbound}.  By Cauchy's inequality, its contribution to the
local energy inequalities is controlled by
\[
  \eta\int_D|\nabla w_\eps|^2
  +
  C_\eta
  \left(1+\int_D|\zeta|^2+\int_D|\nabla\zeta|^2\right),
\]
with \(\eta>0\) chosen small.  Thus it can be absorbed into the
left-hand side of the local estimates, up to a constant depending only
on \(M,\lambda_{\rm ell},\Lambda_{\rm ell}\).

For the half-disc, one subtracts the constant boundary value by writing
\(w_\eps=1+\widetilde w_\eps\) and reflects \(\widetilde w_\eps\) evenly
across the flat boundary.  The reflected coefficient matrix remains
uniformly elliptic and satisfies the same scaled derivative bounds.
The reflected equation has the same form as
\eqref{pf:eq:weighted-bos-eqn}, with a perturbation still satisfying
\eqref{pf:eq:Rbound}.  Applying the interior estimate to the reflected
disc gives the half-disc estimate.  Finally, scaling gives the statement
for arbitrary smaller charts \(D'\Subset D\).
\end{proof}

\begin{corollary}[Good-region weighted BOS estimate]
\label{pf:prop:weighted-bos-good}
Let \(D\) be either the unit disc \(B_1\) or the upper half-disc
\(B^+_1\).  Let \(D'\Subset D\cup(\partial D\cap\{x_2=0\})\) be a
concentric smaller disc or half-disc.  Let \(w_\eps:D\to\mathbb C\)
satisfy the hypotheses of Proposition~\ref{pf:prop:weighted-bos}.

There exist constants
\[
  \sigma_0\in(0,1/2),\qquad \beta_{\rm dec}>0,\qquad C>0,
\]
depending only on \(D'\), \(\lambda_{\rm ell}\), and \(\Lambda_{\rm ell}\), such that if
\begin{equation}\label{pf:eq:good-region-assumption}
  |w_\eps|\ge1-\sigma_0
  \quad\text{on }D,
\end{equation}
then
\begin{equation}\label{pf:eq:good-region-bos}
  \int_{D'} |\nabla |w_\eps||^2\,dx
  +
  \frac1{\eps^2}
  \int_{D'}(1-|w_\eps|^2)^2\,dx
  \le
  C\eps^{\beta_{\rm dec}}
  \left[
  \int_D|\nabla w_\eps|^2\,dx
  +
  \frac1{\eps^2}\int_D(1-|w_\eps|^2)^2\,dx
  +
  1
  \right].
\end{equation}
In particular, if
\[
  \int_D|\nabla w_\eps|^2\,dx
  +
  \frac1{\eps^2}\int_D(1-|w_\eps|^2)^2\,dx
  \le
  M\ln\frac1\eps ,
\]
then
\[
  \frac1{\eps^2}
  \int_{D'}(1-|w_\eps|^2)^2\,dx
  \le
  C_M\eps^{\beta_{\rm dec}}\ln\frac1\eps
  \rightarrow0 .
\]
\end{corollary}

\begin{proof}
In the unweighted flat case
$
  A_\eps=I, B_\eps=1, \mathcal R_\eps=0,
$
this is the good-region estimate obtained from
Bethuel--Orlandi--Smets, Proposition~A.4, as recalled in
\cite[Remark~1.2]{BOS}.  More precisely, on a ball where the clearing-out
condition
$
  |u_\eps|\ge1-\sigma_0
$
holds, their estimate gives
\[
  \int_{B_{1/2}}
  \left(
  |\nabla |u_\eps||^2
  +
  \frac{(1-|u_\eps|^2)^2}{\eps^2}
  \right)
  \le
  C\eps^{\beta_{\rm dec}} E_\eps(u_\eps;B_1).
\]
Here
\[
  E_\eps(u_\eps;B_1)
  =
  \int_{B_1}|\nabla u_\eps|^2
  +
  \frac1{\eps^2}\int_{B_1}(1-|u_\eps|^2)^2,
\]
up to harmless normalization constants.

We now explain why the same estimate holds in the present fixed-chart weighted
setting.  The coefficient assumptions imply that the weighted kinetic energy is
uniformly equivalent to the Euclidean kinetic energy:
\[
  \lambda_{\rm ell}\int_D|\nabla w_\eps|^2
  \le
  \int_D A_\eps\nabla w_\eps\cdot\nabla w_\eps
  \le
  \Lambda_{\rm ell}\int_D|\nabla w_\eps|^2 .
\]
Likewise, the potential weighted by \(B_\eps\) is uniformly equivalent to the
unweighted potential:
\[
  \lambda_{\rm ell}
  \frac1{\eps^2}\int_D(1-|w_\eps|^2)^2
  \le
  \frac1{\eps^2}\int_D B_\eps(1-|w_\eps|^2)^2
  \le
  \Lambda_{\rm ell}
  \frac1{\eps^2}\int_D(1-|w_\eps|^2)^2 .
\]

The proof of the good-region estimate in BOS is local.  After rescaling to the
GL scale \(x=x_0+\eps y\), the coefficient assumptions give
\[
  \|\nabla_y A_\eps(x_0+\eps y)\|_{L^\infty}
  +
  \|\nabla_y B_\eps(x_0+\eps y)\|_{L^\infty}
  \le C(\Lambda_{\rm ell}),
\]
so the rescaled coefficients remain uniformly elliptic with uniformly bounded
Lipschitz constants.  Therefore the Caccioppoli inequalities, the modulus
equation estimates, and the local covering argument used in
\cite[Proposition~A.4]{BOS} are stable under this uniformly elliptic
perturbation.

The lower-order term \(\mathcal R_\eps\) is controlled by the assumed bound.  In
the local energy inequalities its contribution is estimated by Cauchy's inequality:
for every \(\eta>0\),
\[
  |\langle\mathcal R_\eps,\zeta\rangle|
  \le
  \eta\int_D|\nabla w_\eps|^2
  +
  C_\eta
  \int_D
  \bigl(|\zeta|^2+|\nabla\zeta|^2+1\bigr),
\]
and the term with \(\eta\) is absorbed into the left-hand side of the local
estimate.  The remaining terms are lower order and are absorbed into the
constant \(C\) in \eqref{pf:eq:good-region-bos}.

In the half-disc case, since \(w_\eps=1\) on the flat boundary, one applies the
same argument after the standard even reflection across the flat boundary.  The
reflected coefficient matrix is still uniformly elliptic and satisfies the same
scaled derivative bounds, and the reflected perturbation satisfies the same type
of estimate.  This gives \eqref{pf:eq:good-region-bos} in the half-disc as well.

Finally, if the local energy is \(O(\ln(1/\eps))\), then
\eqref{pf:eq:good-region-bos} gives
\[
  \frac1{\eps^2}
  \int_{D'}(1-|w_\eps|^2)^2
  \le
  C_M\eps^{\beta_{\rm dec}}\ln\frac1\eps,
\]
which tends to zero.
\end{proof}

\begin{proposition}[Fixed-truncation estimate]
\label{prop:fixed-truncation-main}
Let $U_\eps$ be a reduced
admissible family of  critical points of
\eqref{eq:GL_energy} defined in Definition \ref{def:reduced-admissible-main} For every fixed $\delta>0$, there
are constants $c_\delta,C_\delta>0$, independent of $\eps$, such that
\begin{equation}
  c_\delta\le\rho_\eps\le1,
  \qquad
  \eps\|\nabla\rho_\eps\|_{L^\infty(\Omega_\delta)}\le C_\delta,
  \label{eq:fixed-scalar-main}
\end{equation}
and
\begin{equation}
  \frac1{\eps^2}\int_{\Omega_\delta}
  (1-|v_\eps|^2)^2\,\dd x\le C_\delta.
  \label{eq:fixed-potential-main}
\end{equation}
\end{proposition}
\begin{proof}
The scalar bounds in \eqref{eq:fixed-scalar-main} were proved in
Lemma~\ref{pf:lem:scalar-fixed-truncation}.  It remains to prove
\eqref{eq:fixed-potential-main}.  On $\Omega_\delta$ this estimate is equivalent
to the weighted form
\begin{equation}\label{pf:eq:Pdelta-weighted}
  \int_{\Omega_\delta}V_\eps(v_\eps)\,dx
  =\frac1{4\eps^2}\int_{\Omega_\delta}\rho_\eps^4(1-|v_\eps|^2)^2\,dx
  \le C_\delta .
\end{equation}
By Lemma~\ref{pf:lem:scalar-fixed-truncation}, \(\rho_\eps\) is uniformly
positive on \(\Omega_\delta\).  Hence the reduced logarithmic energy
bound \eqref{eq:reduced-log-bound-main} gives
\begin{equation}\label{pf:eq:standard-log-on-truncation}
  \int_{\Omega_\delta}|\nabla v_\eps|^2\,dx
  +
  \frac1{\eps^2}
  \int_{\Omega_\delta}(1-|v_\eps|^2)^2\,dx
  \le C_\delta\ln\frac1\eps .
\end{equation}

Let
$
  \mathcal K_\delta:=\overline{\Omega_\delta}\cap \Omega$.
Notice that \(\mathcal K_\delta\) includes the artificial circular arcs
\(\partial S_j(\delta)\cap \Omega\), but does not include the physical
boundary points of \(\partial \Omega\).  Cover \(\mathcal K_\delta\) by finitely
many chart pairs
$D_m'\Subset D_m, m=1,\ldots,M_\delta$.
Each pair is one of the following two types.

First, an interior chart:
\[
  D_m=B_{2r_m}(x_m)\Subset \Omega\setminus V,
  \qquad
  D_m'=B_{r_m}(x_m).
\]
The artificial circular arcs \(\partial S_j(\delta)\cap \Omega\) are
covered by such interior charts, because they are not boundary for the
equation.
On an interior chart, \(v_\eps\) solves
\[
  -\operatorname{div}(a_\eps\nabla v_\eps)
  =
  \eps^{-2}b_\eps v_\eps(1-|v_\eps|^2),
  \qquad
  a_\eps=\rho_\eps^2,\qquad b_\eps=\rho_\eps^4 .
\]
By Lemma~\ref{pf:lem:scalar-fixed-truncation}, the coefficients satisfy
the ellipticity and scaled derivative assumptions of
Proposition~\ref{pf:prop:weighted-bos}.  The logarithmic energy hypothesis
of Proposition~\ref{pf:prop:weighted-bos} follows from
\eqref{pf:eq:standard-log-on-truncation}.  Therefore
\[
  \frac1{\eps^2}
  \int_{D_m'}(1-|v_\eps|^2)^2\,dx
  \le C_m .
\]

Second, a boundary half-disc chart along a compact smooth side piece:
there is a \(C^2\) diffeomorphism
$
  \Psi_m:B^+_{2r_m}\to \Omega\cap U_m
$
where $U_m\cap\partial \Omega\neq\emptyset$ and away from $V$ such that
$D_m=\Psi_m(B^+_{2r_m})$,
$D_m'=\Psi_m(B^+_{r_m})$,
and \(\Psi_m(\{y_2=0\})=\partial \Omega\cap U_m\).
Work in one such
chart and write
\[
  V_\eps(y):=v_\eps(\Psi_m(y)).
\]
Let
$J(y):=|\det D\Psi_m(y)|$,
  $\mathsf A(y):=D\Psi_m(y)^{-1}D\Psi_m(y)^{-T}$.
Then \(V_\eps\) satisfies
\begin{equation}\label{pf:eq:flattened-v-equation}
  -\operatorname{div}_y(A_\eps(y)\nabla_y V_\eps)
  =
  \eps^{-2}B_\eps(y)V_\eps(1-|V_\eps|^2),
\end{equation}
where
\[
  A_\eps(y)
  =
  \rho_\eps^2(\Psi_m(y))J(y)\mathsf A(y),
  \qquad
  B_\eps(y)
  =
  \rho_\eps^4(\Psi_m(y))J(y).
\]
Since $\Psi_m$ is \(C^2\) diffeomorphism, the matrix \(A_\eps\) is uniformly elliptic on the fixed chart, and
Lemma~\ref{pf:lem:scalar-fixed-truncation} gives
\[
  \eps\|\nabla A_\eps\|_{L^\infty}
  +
  \eps\|\nabla B_\eps\|_{L^\infty}
  \le C_\delta .
\]

Choose a local $C^2$ phase lift of $g_{\rm ph}$ on this boundary subarc
and extend it to a function \(\chi\in C^2(B^+_{2r_m})\), so that
$e^{\ii\chi(y_1,0)}=g_{\rm ph}(\Psi_m(y_1,0))$.
Define
\[
  w_\eps(y):=e^{-\ii\chi(y)}V_\eps(y).
\]
Then
$|w_\eps|=|V_\eps|=|v_\eps|$,
$w_\eps=1$ on $y_2=0$.
Since
$V_\eps=e^{\ii\chi}w_\eps$,
$\nabla V_\eps
  =
  e^{\ii\chi}\bigl(\nabla w_\eps+\ii w_\eps\nabla\chi\bigr)$,
equation \eqref{pf:eq:flattened-v-equation}, after multiplication by
\(e^{-\ii\chi}\), becomes
\begin{equation}\label{pf:eq:gauged-halfdisc-equation}
  -\operatorname{div}(A_\eps\nabla w_\eps)
  =
  \eps^{-2}B_\eps w_\eps(1-|w_\eps|^2)
  +
  \mathcal R_\eps,
\end{equation}
where
\begin{equation}\label{pf:eq:gauge-remainder}
  \mathcal R_\eps
  =
  i\,\operatorname{div}(A_\eps w_\eps\nabla\chi)
  +
  i\,A_\eps\nabla\chi\cdot\nabla w_\eps
  -
  (A_\eps\nabla\chi\cdot\nabla\chi)w_\eps .
\end{equation}
Equivalently, for every test function \(\zeta\),
\[
  \langle\mathcal R_\eps,\zeta\rangle
  =
  -\ii\int A_\eps w_\eps\nabla\chi\cdot\nabla\zeta
  +
  \ii\int A_\eps\nabla\chi\cdot\nabla w_\eps\,\zeta
  -
  \int (A_\eps\nabla\chi\cdot\nabla\chi)w_\eps\zeta .
\]
Because
\[
  |w_\eps| = |v_\eps|\le1,\qquad
  \|A_\eps\|_{L^\infty}\le C_\delta,\qquad
  \|\chi\|_{C^2}\le C_\delta,
\]
this satisfies the perturbation bound in
Proposition~\ref{pf:prop:weighted-bos}.  Moreover,
\[
  |\nabla w_\eps|^2
  \le
  C\bigl(|\nabla V_\eps|^2+1\bigr),
\]
so the logarithmic energy bound
\eqref{pf:eq:standard-log-on-truncation} is preserved in the chart.
Thus Proposition~\ref{pf:prop:weighted-bos} gives
\[
  \frac1{\eps^2}
  \int_{D_m'}(1-|v_\eps|^2)^2\,dx
  =
  \frac1{\eps^2}
  \int_{\Psi_m^{-1}(D_m')}(1-|w_\eps|^2)J\,dy
  \le C_m .
\]

Summing over the finite covering gives
\eqref{eq:fixed-potential-main}.
Since \(0\le\rho_\eps\le1\), the weighted estimate
\eqref{pf:eq:Pdelta-weighted} follows immediately.  
\end{proof}

\subsection{Bad discs, clearing-out, and compactness}
Fix $\eta_0\in(0,1/2)$ and set the bad discs as
\begin{equation}\label{pf:eq:bad-set}
  \mathcal B_\eps:=\{x\in \Omega: |v_\eps(x)|\le1-\eta_0\}.
\end{equation}
The proof below follows the order of BBH Chapter~X on each fixed truncation
$\Omega_\delta$: potential estimate, bad discs, linear divergence-form estimates for the phase, and convergence.  The only polygonal operation is the final diagonal passage $\delta\rightarrow0$, where all bad centres converging to vertices are absorbed into shrinking vertex sectors.

\begin{lemma}[Bad discs on fixed truncations]\label{pf:lem:fixed-truncation-bad-discs}
For every fixed \(\delta>0\) there are constants
\(L_\delta>0\) and \(M_\delta\in\mathbb N\), independent of
\(\eps\), such that
$
  \mathcal B_\eps\cap \Omega_\delta
$
is covered by at most \(M_\delta\) physical relative balls
\[
  B_\Omega(x_{\eps,h},L_\delta\eps)
  :=
  \Omega\cap B(x_{\eps,h},L_\delta\eps),
  \qquad h=1,\ldots,M_\delta.
\]
These relative balls are full balls in the interior and flattened half-balls
along smooth side pieces.
\end{lemma}

\begin{proof}
On the slightly larger fixed truncation \(\Omega_{\delta/2}\),
Lemma~\ref{pf:lem:scalar-fixed-truncation} gives
\[
  \rho_\eps\ge c_\delta>0 .
\]
Let \(x_0\in B_\eps\cap \Omega_\delta\).  For \(\eps\) small enough
depending on \(\delta\), the physical relative ball
\[
  B_\Omega(x_0,r_\delta\eps):=\Omega\cap B(x_0,r_\delta\eps)
\]
is contained in \(\Omega_{\delta/2}\).  If \(x_0\) is close to a smooth
side, this is a flattened half-ball; if \(x_0\) is close only to an
artificial circular arc \(\partial S_j(\delta)\cap \Omega\), it is an
interior ball, since the artificial arc is not a
boundary.

By Lemma~\ref{pf:lem:gradscale}, after choosing \(r_\delta>0\) sufficiently
small depending only on \(\eta_0\) and \(\delta\), one has
\[
  |v_\eps|\le1-\eta_0/2
  \quad\text{on }B_\Omega(x_0,r_\delta\eps).
\]
Hence
\[
  (1-|v_\eps|^2)^2\ge c(\eta_0)>0
  \quad\text{on }B_\Omega(x_0,r_\delta\eps).
\]
The relative ball has area comparable to \(\eps^2\), uniformly for fixed
\(\delta\), both in the full-ball and half-ball cases.  Therefore
\[
  \frac1{\eps^2}
  \int_{B_\Omega(x_0,r_\delta\eps)}
  (1-|v_\eps|^2)^2\,dx
  \ge c(\eta_0,\delta)>0 .
\]
The potential estimate \eqref{eq:fixed-potential-main}, applied on \(\Omega_{\delta/2}\),
bounds the number of pairwise disjoint such bad relative balls.  Vitali's
selection lemma or the $5r$-covering lemma \cite[Theorem~2.1]{Mattila1999} gives a countable pairwise disjoint subfamily whose fixed enlargement covers
\(B_\eps\cap \Omega_\delta\).  
Replacing \(r_\delta\) by this enlarged
radius gives the required \(L_\delta\) and \(M_\delta\).  
This is
the usual BBH bad-disc selection, applied only on the fixed truncation.
\end{proof}

\begin{proposition}[BBH compactness on a fixed truncation]
\label{pf:prop:fixed-truncation-compactness}
Fix $\delta>0$. For every sequence $\eps_\ell\rightarrow0$,
there exist a subsequence, still denoted by $\eps_\ell$, a
finite set
$
A_\delta\subset \Omega_\delta,
$
and a smooth map
$
u_\delta:\Omega_\delta\setminus A_\delta\rightarrow S^1
$
such that
\begin{equation}
\label{pf:eq:fixed-trunc-conv}
v_{\eps_\ell}\rightarrow u_\delta
\qquad\text{in }
C^k_{\mathrm{loc}}(\Omega_\delta\setminus A_\delta)
\quad\forall k\geq0.
\end{equation}
The limit $u_\delta$ is an $S^1$-valued harmonic map in
$\Omega_\delta\setminus A_\delta$, and its trace on every compact
smooth-side subarc $\Gamma_k\cap \overline{\Omega}_{\delta}$ is $g_{\rm ph}$.
\end{proposition}

\begin{proof}
Fix $\delta>0$ and work 
on the truncation $\Omega_{\delta}$. By
Lemma~\ref{pf:lem:scalar-fixed-truncation}, for all sufficiently small
$\eps$,
\begin{equation}
\label{pf:eq:fixed-coeff-bounds-proof}
c_\delta\leq\rho_\eps\leq1,
\qquad
\eps
\|\nabla\rho_\eps\|_{L^\infty(\Omega_{\delta})}
\leq C_\delta.
\end{equation}
In particular, the coefficients
\[
a_\eps:=\rho_\eps^2,
\qquad
b_\eps:=\rho_\eps^4
\]
are uniformly elliptic on every fixed-truncation $\Omega_\delta$. Moreover,
Proposition~\ref{prop:fixed-truncation-main} gives
\begin{equation}
\label{pf:eq:fixed-potential-proof}
\frac1{\eps^2}
\int_{\Omega_{\delta}}
(1-|v_\eps|^2)^2\,dx
\leq C_\delta.
\end{equation}

With the above conditions, we use below the compactness mechanism developed for
critical Ginzburg--Landau solutions in
\cite[Theorems~X.1--X.3]{BBH}. More precisely,
\cite[Theorem~X.1]{BBH} gives the subcritical
$W^{1,p}$ estimates, $1<p<2$; Step~1 in the proof of
\cite[Theorem~X.2]{BBH} gives the strong $H^1$ convergence away
from the limiting bad points; and Steps~2--3 of the same proof give
the local uniform and $C^k$ convergence. The corresponding argument
near a smooth boundary portion is contained in
\cite[Theorem~X.3]{BBH}.

Choose the bad-set threshold $\eta_0$ no larger than the constant
$\sigma_0$ in Corollary~\ref{pf:prop:weighted-bos-good}. By
Lemma~\ref{pf:lem:fixed-truncation-bad-discs}, the set
$
\mathcal B_\eps\cap \Omega_{\delta}
=
\{x\in \Omega_{\delta}:|v_\eps(x)|
\leq1-\eta_0\}
$
is covered by at most $M_\delta$ relative balls of radius
$L_\delta\eps$. After passing to a subsequence and
relabeling the centres, every sequence of centres that remains in a
compact subset of $\Omega_\delta$ converges. Let
$A_\delta$
be the finite set of their distinct limits lying in
$\Omega_\delta$. Since the radii of the bad balls tend to zero, every
accumulation point in $\Omega_\delta$ of the bad set belongs to
$A_\delta$.

The local $W^{1,p}$ estimate used in
\cite[Theorem~X.1]{BBH}
applies on a
finite covering of $\Omega_\delta$ by interior discs and flattened
boundary half-discs. In the latter charts one first gauges away the
boundary phase. We therefore obtain, for every $1<p<2$,
\begin{equation}
\label{pf:eq:fixed-W1p}
\sup_\ell
\|v_{\eps_\ell}\|_{W^{1,p}(\Omega_\delta)}
\leq C_{\delta,p}.
\end{equation}
The artificial circular arcs are treated by interior charts, since
they are not boundaries for the original equation.

Let
$
\mathcal K\Subset \Omega_\delta\setminus A_\delta.
$
Choose finitely many simply connected nested patches
$
D'\Subset D\Subset
\Omega_{\delta}\setminus A_\delta
$
covering $\mathcal K$. For all sufficiently large $\ell$, the set $D$ is
disjoint from the enlarged bad balls. Hence
$
|v_{\eps_\ell}|\geq1-\eta_0$ in $D$.
We may consequently write
\[
v_{\eps_\ell}
=
s_{\eps_\ell}e^{\ii\psi_{\eps_\ell}},
\qquad
s_{\eps_\ell}:=
|v_{\eps_\ell}|
\geq1-\eta_0.
\]
The lift is chosen locally; equivalently, in the global BBH
decomposition one first subtracts the singular phases generated by
the bad-disc degrees. Since $D$ avoids $A_\delta$, no singular phase
remains in $D$.

The imaginary and real parts of the reduced equation in \eqref{eq:reduced-equation-main} are respectively
\begin{equation}
\label{pf:eq:phase-fixed-proof}
\operatorname{div}
\left(
\rho_{\eps_\ell}^2
s_{\eps_\ell}^2
\nabla\psi_{\eps_\ell}
\right)
=0
\end{equation}
and
\begin{equation}
\label{pf:eq:modulus-fixed-proof}
-\operatorname{div}
\left(
\rho_{\eps_\ell}^2
\nabla s_{\eps_\ell}
\right)
=
\eps_\ell^{-2}
\rho_{\eps_\ell}^4
s_{\eps_\ell}
(1-s_{\eps_\ell}^2)
-
\rho_{\eps_\ell}^2
s_{\eps_\ell}
|\nabla\psi_{\eps_\ell}|^2.
\end{equation}
The phase coefficient \eqref{pf:eq:phase-fixed-proof} satisfies
\[
c_\delta^2(1-\eta_0)^2
\leq
\rho_{\eps_\ell}^2s_{\eps_\ell}^2
\leq1.
\]

Corollary~\ref{pf:prop:weighted-bos-good} applied to
\eqref{pf:eq:modulus-fixed-proof}
gives
\begin{equation}
\label{pf:eq:modulus-decay-fixed}
\int_{D'}
|\nabla s_{\eps_\ell}|^2\,dx
+
\frac1{\eps_\ell^2}
\int_{D'}
(1-s_{\eps_\ell}^2)^2\,dx
\rightarrow0.
\end{equation}

We now repeat Step~1 in the proof of
\cite[Theorem~X.2]{BBH}. The $W^{1,p}$ estimate
\eqref{pf:eq:fixed-W1p}, followed by the Caccioppoli estimate for the
uniformly elliptic phase equation
\eqref{pf:eq:phase-fixed-proof}, gives a local $H^1$ bound for the
normalized phase, and 
passing to a further subsequence, we have
$
\psi_{\eps_\ell}
\rightharpoonup\psi_\delta$
weakly in $H^1(D')$.
Since $D$ is compactly contained in $\Omega$, the scalar profile
satisfies
$
\rho_{\eps_\ell}\rightarrow1
$ locally uniformly in $D$.

Together with \eqref{pf:eq:modulus-decay-fixed}, this implies
$
\rho_{\eps_\ell}^2
s_{\eps_\ell}^2
\rightarrow1$ strongly in $L^q(D')$ for finite $q$.
  Repeating the test-function argument of
Step~1 in \cite[Theorem~X.2]{BBH}, applied to the difference between
\eqref{pf:eq:phase-fixed-proof} and its limiting equation, therefore
yields
$
\psi_{\eps_\ell}
\rightarrow\psi_\delta$ strongly in $H^1_{\mathrm{loc}}(D)$.
At the same time, \eqref{pf:eq:modulus-decay-fixed} gives
$
s_{\eps_\ell}
\rightarrow1$
strongly in $H^1_{\mathrm{loc}}(D)$.
Consequently,
\begin{equation}
\label{pf:eq:strong-H1-fixed}
v_{\eps_\ell}
\rightarrow
e^{\ii\psi_\delta}
\qquad\text{strongly in }H^1_{\mathrm{loc}}(D).
\end{equation}

Passing to the limit in \eqref{pf:eq:phase-fixed-proof} gives
$
\Delta\psi_\delta=0
$ in $D$.
Hence the local limit
$
u_\delta:=e^{\ii\psi_\delta}
$
is an $S^1$-valued harmonic map. The limits obtained on overlapping
patches agree, since they are limits of the same subsequence, and
therefore define a smooth harmonic map on
$\Omega_\delta\setminus A_\delta$.

Once the strong $H^1$ convergence
\eqref{pf:eq:strong-H1-fixed} is known, the degree-zero regularity
argument in Steps~2--3 of
\cite[Theorem~X.2]{BBH}, together with
\cite[Appendix~I, Remark~A.1]{BBH}, applies on each interior
patch. The scalar estimates and standard interior elliptic estimates
give smooth convergence of $\rho_{\eps_\ell}$ to $1$ on
compact subsets. Elliptic bootstrapping in the reduced equation then
yields
\[
v_{\eps_\ell}
\rightarrow u_\delta
\qquad \text{in 
$C^k_{\mathrm{loc}}
(\Omega_\delta\setminus A_\delta)$\qquad
$\forall k\geq0$}.
\]

It remains to identify the boundary trace. Let $H$ be a flattened
half-disc centred on a compact smooth-side subarc and let $\chi$ be a
$C^2$ extension of the boundary phase. Set
$
w_{\eps_\ell}
:=
e^{-\ii\chi}v_{\eps_\ell}.
$
Then
$
w_{\eps_\ell}=1
$
on the flat part of the boundary. The half-disc version of the
$W^{1,p}$ estimate gives weak compactness in $W^{1,p}(H)$ for some
$p>1$. Since the trace operator is continuous on $W^{1,p}(H)$, the
weak limit has trace $1$. Undoing the gauge gives
$
\operatorname{Tr}u_\delta=g_{\rm ph}
$
on the smooth-side subarc. This proves the stated trace condition.

\end{proof}

\begin{theorem}[Clearing-out and polygonal compactness]
\label{thm:clearing-main}
For every sequence $\eps_\ell\rightarrow0$, there are a subsequence, a
finite set
\[
  A=\{a_1,\ldots,a_N\}\Subset \Omega,
\]
such that,
\begin{equation}
  v_{\eps_\ell}\rightarrow u_*
  \quad\text{in }C^k_{\rm loc}(\Omega\setminus(A\cup V))
  \quad\forall k\ge0,
  \label{eq:compactness-main}
\end{equation}
where $u_*$ is a smooth $S^1$-valued harmonic map.  Each stabilized
interior degree
\begin{equation}
  d_i=\deg(u_*,\partial B(a_i,r))
  \label{eq:interior-degree-main}
\end{equation}
is a nonzero integer.
\end{theorem}

\begin{proof}

\medskip
\noindent
\textbf{Step 1: relative exhaustion by fixed truncations.}
Choose a decreasing sequence \(\delta_m\rightarrow0\) and set
\[
  \Omega_m:=\Omega_{\delta_m}
  =
  \Omega\setminus\bigcup_{j=1}^K S_j(\delta_m).
\]
Consequently, we have
$
  \bigcup_{m=1}^\infty \Omega_m=\Omega\setminus V,
$
and for every compact set
$
  \mathcal K\Subset\overline \Omega\setminus V
$
there exists \(m\) such that
$
  \mathcal K\cap \Omega\subset \Omega_m.
$
Applying Lemma~\ref{pf:lem:fixed-truncation-bad-discs} on the fixed truncation \(\Omega_m\),
there exist constants \(L_m>0\) and \(M_m<\infty\), independent of
\(\eps\), such that
\begin{equation}\label{pf:eq:fixed-m-bad-cover}
  \mathcal B_\eps\cap \Omega_m
  \subset
  \bigcup_{h=1}^{M_m}
  B_\Omega(x^{(m)}_{\eps,h},L_m\eps).
\end{equation}

On the same fixed truncation, according to
Proposition~\ref{pf:prop:fixed-truncation-compactness},
there is a finite set
$
  A^{(m)}\subset \Omega_m
$
consisting of the non-removable bad limits in \(\Omega_m\), and a smooth
\(S^1\)-valued harmonic map \(u_*^{(m)}\) on
\(\Omega_m\setminus A^{(m)}\), such that
\begin{equation}\label{pf:eq:fixed-m-convergence}
  v_\eps\to u_*^{(m)}
  \quad\text{in }C^k_{\rm loc}(\Omega_m\setminus A^{(m)})
  \quad\forall k\ge0.
\end{equation}

\medskip
\noindent
\textbf{Step 2: diagonal subsequence and the local bad-limit set.}
We now extract one subsequence which works for all fixed truncations.  
We take the diagonal subsequence, relabelled by
\(\eps_\ell\), which has the property that for every fixed \(m\),
\eqref{pf:eq:fixed-m-convergence} holds along \(\eps_\ell\) for all sufficiently large
\(\ell\).

Define the local bad-limit set away from the vertices by
\[
  A_{\rm loc}
  :=
  \left\{
  a\in\overline \Omega\setminus V:\,
  \exists\,\ell_s\to\infty,\ 
  \exists\,y_{\ell_s}\in\mathcal B_{\eps_{\ell_s}}
  \text{ such that } y_{\ell_s}\to a
  \right\}.
\]
Thus \(A_{\rm loc}\) is the set of all accumulation points, away from \(V\), of
points where \(|v_{\eps_\ell}|\le1-\eta_0\).

We first prove local finiteness in the interior.  Let
$
  \mathcal K\Subset \Omega\setminus V.
$
Choose \(m\) so large that
$
  \mathcal K\Subset \Omega_m.
$
We claim that
\begin{equation}\label{pf:eq:Aloc-subset-Am}
  A_{\rm loc}\cap\mathcal K\subset A^{(m)}.
\end{equation}
Indeed, if
$
  a\in A_{\rm loc}\cap\mathcal K
$ and $a\notin A^{(m)},
$
then there is a ball
$
  B(a,r)\Subset \Omega_m\setminus A^{(m)}.
$
By \eqref{pf:eq:fixed-m-convergence},
$
  |v_{\eps_\ell}|\to1
$ uniformly on $B(a,r/2)$.
Therefore for \(\ell\) sufficiently large,
$
  |v_{\eps_\ell}|>1-\eta_0
$ on $B(a,r/2)
$,
contradicting the fact that \(a\) is the limit of bad points.  This proves
\eqref{pf:eq:Aloc-subset-Am}.  Since \(A^{(m)}\) is finite, \(A_{\rm loc}\cap\mathcal K\)
is finite.  Hence \(A_{\rm loc}\cap \Omega\) is locally finite in
\(\Omega\setminus V\).  

\medskip
\noindent
\textbf{Step 3: non-zero degree at interior non-removable points.}
Let \(a\in A_{\rm loc}\cap \Omega\).  Choose \(m\) and \(r>0\) such that
$
  \overline{B(a,2r)}
  \Subset
  \Omega_m
$
and
$
  \partial B(a,r)\cap A_{\rm loc}=\varnothing .
$
This is possible by the local finiteness just proved.  For all sufficiently large
\(\ell\), there are no bad points on \(\partial B(a,r)\), so
$
  |v_{\eps_\ell}|\ge1-\eta_0
  \quad\text{on }\partial B(a,r).
$
Hence the degree
$
  \nu(a):=\deg(v_{\eps_\ell},\partial B(a,r))
$
is well-defined and independent of \(\ell\) for all large \(\ell\).

We claim that
$
  \nu(a)\ne0.
$
By Step~2, \(a\in A^{(m)}\) for every \(m\) with
\(B(a,2r)\Subset \Omega_m\).  Recall that \(A^{(m)}\) was chosen as the minimal
non-removable singular set produced by the fixed-truncation BBH Chapter X
compactness theorem.  If \(\nu(a)=0\), the local phase has no singular topological
part.  In that case the zero-degree local compactness mechanism applies: after
writing a single-valued phase on \(\partial B(a,r)\), the phase equation
$
  \operatorname{div}
  \bigl(\rho_{\eps_\ell}^2|v_{\eps_\ell}|^2\nabla\psi_{\eps_\ell}\bigr)=0
$
is uniformly elliptic on the fixed truncation, and the BBH--Stampacchia estimates
give strong local \(H^1\) compactness in \(B(a,r/2)\).  Then the degree-zero
convergence theorem of BBH Appendix~I, together with Remark~A.1 there for
solutions once strong \(H^1\) convergence is known, gives
$
  |v_{\eps_\ell}|\to1
  \quad\text{uniformly on }B(a,r/4).
$
This contradicts \(a\in A_{\rm loc}\).  Hence \(\nu(a)\ne0\).

\medskip
\noindent
\textbf{Step 4: finiteness of the interior bad-limit set.}
Let
$
  a^{(1)},\ldots,a^{(L)}
$
be any finite subset of distinct points of \(A_{\rm loc}\cap \Omega\).  By the
local finiteness proved in Step~2, we may choose pairwise
disjoint balls
$
  B(a^{(s)},2r_s)\Subset \Omega\setminus V,
 s=1,\ldots,L,
$
such that
$
  \partial B(a^{(s)},r_s)\cap A_{\rm loc}=\varnothing.
$
By Step~3, the stabilized degrees
$
  \nu_s
  :=
  \deg(v_{\eps_\ell},\partial B(a^{(s)},r_s))
$
are non-zero integers.

On each ball \(B(a^{(s)},2r_s)\), the scalar profile satisfies
$
  \rho_{\eps_\ell}\to1
$ smoothly on compact subsets
because the ball is compactly contained in \(\Omega\).  Therefore the standard
BBH annular lower bound for a vortex of degree \(\nu_s\) gives
\[
  F_{\eps_\ell}
  \bigl(v_{\eps_\ell};B(a^{(s)},r_s)\bigr)
  \ge
  \bigl(\pi\nu_s^2-o(1)\bigr)
  \ln\frac1{\eps_\ell}
  -
  C_s .
\]
Summing over \(s=1,\ldots,L\), and using the disjointness of the balls, we get
\[
  F_{\eps_\ell}(v_{\eps_\ell})
  \ge
  \left(
  \pi\sum_{s=1}^L\nu_s^2-o(1)
  \right)
  \ln\frac1{\eps_\ell}
  -
  C_L .
\]
On the other hand, reduced admissibility gives
$
  F_{\eps_\ell}(v_{\eps_\ell})
  \le
  C\ln\frac1{\eps_\ell}.
$
Dividing by \(\ln(1/\eps_\ell)\) and letting \(\ell\to\infty\), we obtain
$
  \pi\sum_{s=1}^L\nu_s^2\le C.
$
Since every \(\nu_s\ne0\), it follows that
$
  L\le \frac{C}{\pi}.
$
The finite subset was arbitrary.  Hence
\begin{equation}
  A_{\rm int}:=A_{\rm loc}\cap \Omega
  \quad\text{is finite}.
  \label{pf:eq:interior-bad-set-finite}
\end{equation}

\medskip
\noindent
\textbf{Step 5: no bad limits on smooth sides.}
We next exclude bad limits on \(\partial \Omega\setminus V\).  Suppose
$
  p\in A_{\rm loc}\cap(\partial \Omega\setminus V).
$
By \eqref{pf:eq:interior-bad-set-finite}, choose a flattened half-disc
$
  H_{2r}(p):=\Omega\cap B(p,2r)
$
contained in a compact smooth-side chart and avoiding $A_{\rm int}$.
Let \(\chi\) be a \(C^2\) extension of the boundary phase in this chart and set
$
  w_{\eps_\ell}:=e^{-\ii\chi}v_{\eps_\ell}.
$
Then
$
  w_{\eps_\ell}=1
  \quad\text{on the flat part of the boundary}.
$
Thus the relative boundary degree is zero.  The half-disc version of the
degree-zero convergence theorem, namely BBH Appendix~I, Theorem~A.3, combined
with the fixed-truncation Chapter~X \(H^1\)-compactness estimates in the
flattened and gauged chart, gives
$
  |w_{\eps_\ell}|\to1
  \quad\text{uniformly on }H_{r/2}(p).
$
Since
$
  |w_{\eps_\ell}|=|v_{\eps_\ell}|,
$
this contradicts \(p\in A_{\rm loc}\).  Therefore
$
  A_{\rm loc}\cap(\partial \Omega\setminus V)=\varnothing.
$
Combining this with Step~2 gives
$
  A_{\rm loc}=A_{\rm int}\subset \Omega,
$
and \eqref{pf:eq:interior-bad-set-finite} shows that this set is finite.

We now define
$
  A:=A_{\rm loc}=\{a_1,\ldots,a_N\}.
$
Here \(N=\#A\) is the number of interior limiting bad points.

\medskip
\noindent
\textbf{Step 6: absorption of the remaining bad points into shrinking vertex sectors.}
Choose \(\delta_A>0\) so small that the sectors
$
  S_j(\delta_A), j=1,\ldots,K,
$
are pairwise disjoint and satisfy
$
  A\cap\bigcup_{j=1}^K S_j(\delta_A)=\varnothing .
$
Apply Lemma~\ref{pf:lem:fixed-truncation-bad-discs} to the fixed truncation
\(\Omega_{\delta_A}\).  It gives constants
\[
  L_{\rm bad}:=L_{\delta_A},
  \qquad
  M:=M_{\delta_A},
\]
and centers \(x_{\eps_\ell,h}\), \(1\le h\le M_\ell\le M\), such that
\[
  \mathcal B_{\eps_\ell}\cap \Omega_{\delta_A}
  \subset
  \bigcup_{h=1}^{M_\ell}
  B_\Omega(x_{\eps_\ell,h},L_{\rm bad}\eps_\ell).
\]

For each vertex \(b_j\), define
\[
  r_{\eps_\ell,j}
  :=
  \eps_\ell
  +
  \sup
  \left\{
  |x-b_j|:\,
  x\in
  \mathcal B_{\eps_\ell}\cap S_j(\delta_A),
  \ x\notin
  \bigcup_{h=1}^{M_\ell}
  B_\Omega(x_{\eps_\ell,h},L_{\rm bad}\eps_\ell)
  \right\},
\]
with the convention that the supremum of the empty set is \(0\).  Then every bad
point is covered by
\[
  \bigcup_{h=1}^{M_\ell}
  B_\Omega(x_{\eps_\ell,h},L_{\rm bad}\eps_\ell)
  \cup
  \bigcup_{j=1}^K S_j(r_{\eps_\ell,j}).
\]

We prove that
$
  r_{\eps_\ell,j}\to0
$ for every $j$.
If not, then for some \(j\) there are a number \(c>0\), a subsequence, and points
$
  z_\ell\in
  \mathcal B_{\eps_\ell}\cap S_j(\delta_A)
$
not covered by the \(O(\eps_\ell)\)-balls, such that
$
  |z_\ell-b_j|\ge c.
$
After passing to a further subsequence,
$
  z_\ell\to z_0\in\overline \Omega\setminus V.
$
Since \(z_\ell\) are bad points, \(z_0\in A_{\rm loc}=A\).  But
\(z_\ell\in S_j(\delta_A)\) and \(|z_\ell-b_j|\ge c\), while
$
  A\cap S_j(\delta_A)=\varnothing
$
by the choice of \(\delta_A\).  This is a contradiction.  Thus
\(r_{\eps_\ell,j}\to0\), and the covering
\begin{equation}
  \mathcal B_{\eps_\ell}
  \subset
  \bigcup_{h=1}^{M_\ell}
  B_\Omega(x_{\eps_\ell,h},L_{\rm bad}\eps_\ell)
  \cup
  \bigcup_{j=1}^KS_j(r_{\eps_\ell,j}),
  \qquad M_\ell\le M.
  \label{eq:clearing-cover-main}
\end{equation}
follows.

\medskip
\noindent
\textbf{Step 7: local convergence away from \(A\cup V\).}
Let
$
  \mathcal K\Subset \Omega\setminus(A\cup V).
$
Choose \(m\) such that
$
  \mathcal K\Subset \Omega_m.
$
Since \(A=A_{\rm loc}\), there are no bad accumulation points in \(\mathcal K\).
Hence, for all sufficiently large \(\ell\),
$
  \mathcal B_{\eps_\ell}\cap\mathcal K=\varnothing,
$
or equivalently,
$
  |v_{\eps_\ell}|\ge1-\eta_0
$ on $\mathcal K$.

Moreover, by Step~2 all non-removable
singularities of the fixed-truncation limit in \(\Omega_m\) lie in \(A\).  
On a simply connected patch
$
  D\Subset \Omega\setminus(A\cup V)
$
and for \(\ell\) large, 
the good-region
Bethuel--Orlandi--Smets estimate, in the weighted form of
Proposition~\ref{pf:prop:weighted-bos-good}, gives for \(D'\Subset D\)
\[
  \frac1{\eps_\ell^2}
  \int_{D'}(1-|v_{\eps_\ell}|^2)^2
  \le
  C\eps_\ell^{\beta_{\rm dec}}
  \left(
  \int_D|\nabla v_{\eps_\ell}|^2
  +
  \frac1{\eps_\ell^2}\int_D(1-|v_{\eps_\ell}|^2)^2
  \right).
\]
The bracket is \(O(\ln(1/\eps_\ell))\) by reduced admissibility on the fixed
truncation.  Thus
$
  \frac1{\eps_\ell^2}
  \int_{D'}(1-|v_{\eps_\ell}|^2)^2
  \le
  C\eps_\ell^{\beta_{\rm dec}}\ln\frac1{\eps_\ell}
  \to0.
$
Consequently \(s_{\eps_\ell}\to1\) locally, and elliptic bootstrapping yields the \(C^k_{\rm loc}\) convergence.
Hence, for any $\mathcal K\Subset \Omega\setminus(A\cup V)$
\[
  v_{\eps_\ell}\to u_*
  \quad\text{in }C^k(\mathcal K)
  \quad\forall k\ge0.
\]
The local limits obtained for different \(m\)'s agree on overlaps by uniqueness of
limits, and therefore define a smooth map
\[
  u_*:\Omega\setminus(A\cup V)\to S^1.
\]

Finally, the limit is \(S^1\)-valued and satisfies the harmonic map equation
$
  -\Delta u_*=u_*|\nabla u_*|^2
  \quad\text{in }\Omega\setminus(A\cup V).
$
For each \(a_i\in A\), choose \(r>0\) so small that
$
  \partial B(a_i,r)\Subset \Omega\setminus(A\cup V).
$
The degrees
$
  d_i
  :=
  \deg(u_*,\partial B(a_i,r))
$
are the stabilized degrees of \(v_{\eps_\ell}\) on the same circles, and Step~3
shows that
$
  d_i\ne0.
$
This completes the proof.
\end{proof}
\subsection{Corner jumps and the canonical harmonic map}


\begin{definition}[Proper polygonal harmonic map]\label{def:proper-polygonal-map}
A smooth harmonic map $u:\Om\setminus(A\cup V)\to S^1$ is called proper for the full configuration $\mathcal C = (a_i,d_i;b_j,c_j)$, i.e., interior defects at $a_i$ with winding number $d_i$ and corner defects at $b_j$ with corner winding $c_j = \eta_j/2\pi$, if it has the local forms
\[
  u(z)=\alpha_i\left(\frac{z-a_i}{|z-a_i|}\right)^{d_i}e^{iH_i(z)}
  \quad\text{near }a_i,
\]
and, in sector coordinates $z=b_j+re^{i\theta}$,
\[
  u(b_j+re^{i\theta})=
  \exp\left(i\varphi_j^-+i\frac{\eta_j}{\omega_j}\theta+iH_j(r,\theta)\right)
  \quad\text{near }b_j,
\]
with $H_i,H_j$ harmonic and bounded in the corresponding neighbourhoods.
\end{definition}

\begin{lemma}[Corner jumps]\label{pf:lem:cornerjumps}
Let $u_*$ be the limit in Theorem~\ref{thm:clearing-main}. For each vertex $b_j$ there is an integer $m_j\in\Z$ such that the lifted corner jump of $u_*$ is
$\eta_j=\Phi_j+2\pi m_j=2\pi c_j$, and the degree balance is exactly
\eqref{eq:topological-balance}.
\end{lemma}

\begin{proof}[Reference]
Integrating the lifted phase along the oriented boundary of the punctured domain gives the sum of interior degrees plus the boundary winding and the integer choices of branch across vertices. This is the same topological compatibility used in \cite{ASzeros}.
\end{proof}

\begin{lemma}[Vertex tightness implies properness]
\label{pf:lem:vertex-properness}
The limit $u_*$ in Theorem~\ref{thm:clearing-main}, together with the
vertex logarithmic tightness \eqref{eq:vertex-log-tightness}, has at each
$b_j$ the proper local corner form in
Definition~\ref{def:proper-polygonal-map}.
\end{lemma}

\begin{proof}
Choose a vertex sector containing no interior point of $A$.  Since the
punctured sector is simply connected, $u_*$ has a harmonic phase lift
$\phi_j$.  Select the lift supplied by Lemma~\ref{pf:lem:cornerjumps} and
write
\[
  \phi_j(r,\vartheta_j)
  =\varphi_j^-+\frac{\eta_j}{\omega_j}\vartheta_j
  +H_j(r,\vartheta_j).
\]
The side traces of $H_j$ are bounded and have matching limits at the
vertex.  For every fixed $r>0$, local convergence and
\eqref{eq:vertex-log-tightness} give, by lower semicontinuity,
\[
  \int_{S_j(r_0)\setminus S_j(r)}|\nabla u_*|^2\,\dd x
  \leq C_j\left(1+\ln\frac{r_0}{r}\right).
\]
The explicit angular term has the same logarithmic growth, so $H_j$
satisfies an analogous bound.

In a straight sector, subtract the bounded harmonic extension of the two
side traces and expand the remaining zero-side harmonic function in
sector modes.  Its possible singular terms are
\[
  r^{-n\pi/\omega_j}\sin\left(\frac{n\pi\vartheta_j}{\omega_j}\right),
  \qquad n\geq1.
\]
Every nonzero such term has Dirichlet energy growing as a negative power
of $r$, contradicting the logarithmic bound.  Hence only the bounded
positive sector modes remain.  Under the $C^2$ sector chart, the transformed
equation is a uniformly elliptic perturbation whose coefficients differ
from the constant sector coefficients by $O(r)$; the same corner expansion
and exclusion of negative modes applies.  Thus $H_j$ is bounded at $b_j$,
which is exactly the required properness.
\end{proof}

\begin{proposition}[Canonical polygonal harmonic map]
\label{prop:canonical-main}
For each vertex, there is $m_j\in\Z$ such that
$\eta_j=\Phi_j+2\pi m_j=2\pi c_j$, and the balance law
\eqref{eq:topological-balance} holds.  Hence the  canonical harmonic map (this term comes from BBH theory)
$u_{\mathcal C}$ is well
defined as 
\begin{equation} \label{eq:canonical-map-main}
  u_{\mathcal{C}}(z)=\e^{iH_{\mathcal{C}}}
  \prod_{i=1}^N\left(\frac{z-a_i}{|z-a_i|}\right)^{d_i}
  \prod_{j=1}^K\left(\frac{z-b_j}{|z-b_j|}\right)^{-\frac{\eta_j}{\omega_j}},
\end{equation}
where $\omega_j$ is the inner angle of $b_j$, 
and $H_\mathcal{C}$ is the harmonic solution of the zero-degree boundary problem in 
\begin{equation}\label{eq:canonical-harmonic-correction}
\begin{cases}
  \Delta H_{\mathcal{C}} = 0,
  & \text{in }\Omega,\\[2mm]
  H_{\mathcal{C}} = h_{\mathcal{C}},
  & \text{on }\partial\Omega.
\end{cases}
\end{equation} 
where
\begin{equation}\label{eq:canonical-boundary-lift}
e^{i h_{\mathcal{C}}(z)} = e^{i\varphi_b(z)}\frac{|z-a_1|^{d_1}}{(z-a_1)^{d_1}}\ldots\frac{|z-a_N|^{d_N}}{(z-a_N)^{d_N}}\frac{|z-b_1|^{-\frac{\eta_1}{\omega_1}}}{(z-b_1)^{-\frac{\eta_1}{\omega_1}}}\ldots\frac{|z-b_K|^{-\frac{\eta_K}{\omega_K}}}{(z-b_K)^{-\frac{\eta_K}{\omega_K}}}
\end{equation}
satisfying $u_{\mathcal{C}}(z) = e^{ih_{\mathcal{C}}}\frac{(z-a_1)^{d_1}}{|z-a_1|^{d_1}}\ldots\frac{(z-a_N)^{d_N}}{|z-a_N|^{d_N}} \frac{(z-b_1)^{-\frac{\eta_1}{\omega_1}}}{|z-b_1|^{-\frac{\eta_1}{\omega_1}}}\ldots\frac{(z-b_K)^{-\frac{\eta_K}{\omega_K}}}{|z-b_K|^{-\frac{\eta_K}{\omega_K}}}=e^{i\varphi_b(z)}$ on $\Gamma_j$. 
The map $u_{\mathcal C}$ is the unique proper harmonic map with
the prescribed normalized side traces $\varphi_b$, interior degrees $d_i$, and lifted corner
jumps $\eta_j$.  The limit in Theorem~\ref{thm:clearing-main} is
$u_*=u_{\mathcal C}$.
\end{proposition}
\begin{proof}[Proof of Proposition~\ref{prop:canonical-main}]
The balance law \eqref{eq:topological-balance} is precisely the condition
that the sidewise expressions in \eqref{eq:canonical-boundary-lift} can be
chosen to glue to a single-valued boundary lift.  Their endpoint values
match after the lifted jumps have been subtracted, and their restrictions
to the open sides are $C^2$.  The same cross-side estimate used for $g$ in
Appendix~A therefore gives
$h_{\mathcal C}\in H^{1/2}(\partial \Omega)$.  The Dirichlet problem
\eqref{eq:canonical-harmonic-correction} consequently has a unique weak
solution.  The map in \eqref{eq:canonical-map-main} has the prescribed side
trace, interior degrees, and corner jumps, and is proper in the sense of
Definition~\ref{def:proper-polygonal-map}.  If two proper maps have the
same data, their quotient has zero interior degree, no corner jump, and
unit boundary trace.  Its harmonic phase lift solves the homogeneous
Dirichlet problem and hence vanishes, proving uniqueness.  This is the BBH
canonical-map construction \cite[Corollaries I.1--I.2]{BBH}, with the
fixed sector branches included among the explicitly subtracted factors.

Finally, Theorem~\ref{thm:clearing-main} gives a harmonic $S^1$ map with
the prescribed trace and the interior degrees, while
Lemma~\ref{pf:lem:cornerjumps} supplies its corner branch data and
Lemma~\ref{pf:lem:vertex-properness} supplies properness at the vertices.
Uniqueness therefore identifies that limit with
\eqref{eq:canonical-map-main}.
\end{proof}

\subsection{Stream function and finite-part energy}
For $\rho>0$ smaller than all mutual distances in the configuration, set
\begin{equation}
  \Omega_\rho:=\Omega\setminus\left(
  \bigcup_{i=1}^NB(a_i,\rho)\cup
  \bigcup_{j=1}^K[\Omega\cap B(b_j,\rho)]\right).
  \label{eq:punctured-domain-main}
\end{equation}
Its physical and artificial boundary pieces are
\begin{equation}
  \Gamma_\rho:=\partial \Omega\cap\partial \Omega_\rho,
  \qquad
  \Gamma_{vi}(\rho):=\partial B(a_i,\rho),
  \qquad
  \Gamma_{wj}(\rho):=\partial B(b_j,\rho)\cap \Omega.
  \label{eq:punctured-boundary-main}
\end{equation}
\begin{figure}
    \centering
    \includegraphics[width=0.4\linewidth]{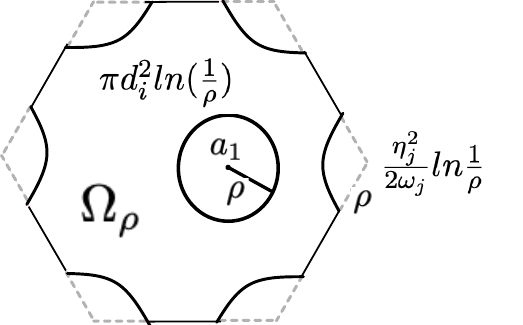}
    \caption{Shrinking-hole definition of the renormalized energy \(W\).  The finite part is computed on
\(\Omega_\rho\), obtained by removing discs around interior defects $a_i$ and sectors
around fixed vertices $b_j$.  Interior holes contribute \(\pi d_i^2\ln(\frac{1}{\rho})\);
corner sectors contribute \(\frac{\eta_j^2}{2\omega_j}\ln(\frac{1}{\rho})\).}
    \label{fig:Omega_rho}
\end{figure}

\begin{proposition}[Linear reduction and finite-part energy]
\label{prop:finite-part-main}
Let $\Phi_\rho$ be the normalized solution, with unknown constants
$D_i(\rho)$ and $C_j(\rho)$, of
\begin{equation}
\left\{
\begin{aligned}
  \Delta\Phi_\rho&=0 &&\text{in }\Omega_\rho,\\
  \Phi_\rho&=D_i(\rho) &&\text{on }\Gamma_{vi}(\rho),\\
  \Phi_\rho&=C_j(\rho) &&\text{on }\Gamma_{wj}(\rho),\\
  \int_{\Gamma_{vi}(\rho)}\partial_\nu\Phi_\rho\,\dd s
  &=-2\pi d_i,\\
  \int_{\Gamma_{wj}(\rho)}\partial_\nu\Phi_\rho\,\dd s
  &=\eta_j,\\
  \partial_\nu\Phi_\rho&=q &&\text{on }\Gamma_\rho,\\
  \int_{\Gamma_\rho}\Phi_\rho\,\dd s&=0.
\end{aligned}
\right.
\label{eq:Phi-rho-problem-main}
\end{equation}
Here $\nu$ is the outward normal of $\Omega_\rho$; in particular it points into
the removed holes on the artificial boundaries.  Let $\Xi_\rho$ be the
class of $H^1(\Omega_\rho;\Sone)$ maps with physical-boundary trace
$g_{\rm ph}$, interior degrees $d_i$, and corner-sector phase increments
$\eta_j$.  Then the BBH linear reduction gives
\begin{equation}
  \inf_{v\in\Xi_\rho}\int_{\Omega_\rho}|\nabla v|^2\,\dd x
  =\int_{\Omega_\rho}|\nabla\Phi_\rho|^2\,\dd x.
  \label{eq:linear-reduction-main}
\end{equation}
\end{proposition}
\begin{proof}
The normalized stream function is the solution of
\eqref{eq:Phi-rho-problem-main}; its uniqueness, with the unknown constant
traces determined by the flux constraints, is the standard linear
shrinking-hole construction \cite{BerestyckiBrezis1980FreeBoundaryPlasma,BBH}.

The purpose of the stream function is to reduce the nonlinear
$S^1$-valued harmonic-map problem to a linear scalar problem. For an
$S^1$-valued map $v$, define its supercurrent by
\[
j(v):=v\times\nabla v.
\]
Locally, if $v=e^{\ii\phi}$, then
\[
j(v)=\nabla\phi.
\]
For an $S^1$-valued harmonic map, the current is divergence-free,
\[
\operatorname{div}j(v)=0,
\]
and, away from the punctures, it is also curl-free. In two dimensions
we therefore represent the canonical current in the form
\[
j(v)=J\nabla\Phi_\rho,
\qquad
J(\xi_1,\xi_2)=(-\xi_2,\xi_1).
\]
Since
\[
\operatorname{curl}(J\nabla\Phi_\rho)=\Delta\Phi_\rho,
\]
the curl-free condition becomes the linear equation
\[
\Delta\Phi_\rho=0
\qquad\text{in }\Omega_\rho.
\]
Moreover,
\[
|\nabla v|^2=|j(v)|^2=|\nabla\Phi_\rho|^2,
\]
so that the Dirichlet energy of the canonical map can be computed
entirely in terms of $\Phi_\rho$.

With the orientation convention $\tau=J\nu$, one has
\[
j(v)\cdot\tau=\partial_\nu\Phi_\rho,
\qquad
j(v)\cdot\nu=-\partial_\tau\Phi_\rho.
\]
Consequently, the prescribed boundary phase gives the Neumann condition
\[
\partial_\nu\Phi_\rho=q
\qquad\text{on }\Gamma_\rho.
\]
On each smooth boundary subarc this is equivalently
$\partial_\nu\Phi_\rho=\partial_\tau\chi$ for any local phase lift
$g_{\rm ph}=e^{\ii\chi}$.
On each artificial boundary component, the condition that
$\Phi_\rho$ be constant is equivalent to the natural zero-normal-current
condition
\[
j(v)\cdot\nu=0.
\]
The prescribed total fluxes across these components encode the
interior vortex degrees and the signed corner phase increments.
The constant traces $D_i(\rho)$ and $C_j(\rho)$ in
\eqref{eq:Phi-rho-problem-main} are not prescribed in advance.  They are
determined together with $\Phi_\rho$ by the flux conditions and the
normalization in that problem.
Thus the stream-function formulation converts the minimization over
$S^1$-valued maps with prescribed topological data into a linear
harmonic problem with constant boundary values and prescribed total
fluxes; this is the basis of the BBH linear reduction
\cite[Theorems~I.1--I.2]{BBH}.

The preceding construction applied to the class $\Xi_\rho$ defined in
Proposition~\ref{prop:finite-part-main} gives
\eqref{eq:linear-reduction-main}.  This is BBH Theorem~I.1 and its
Dirichlet variant, BBH Theorem~I.2 \cite{BBH}.  The artificial boundaries
of $\Omega_\rho$ are smooth, so the same stream-function proof applies.
\end{proof}

We define the limiting stream function $\Phi_0$ as
 \begin{equation}
\label{eq:stream-main}
  \Phi_0(x)=\sum_{i=1}^N d_i\ln|x-a_i|-
  \sum_{j=1}^K\frac{\eta_j}{\omega_j}\ln|x-b_j|+R_0(x),
\end{equation}
and
\begin{align}
&\Delta R_0 = 0,\ \text{on}\ \Omega\label{eq:108}\\
&\frac{\partial R_0}{\partial \nu} =  q + \sum_{j = 1}^K \frac{\eta_j}{\omega_j}  \frac{\partial \ln|z-b_j|}{\partial \nu} - \frac{\partial \sum_{i = 1}^N d_i \ln|z-a_i|}{\partial \nu}\ \text{on}\ \partial \Omega\backslash V,\label{eq:108-1}
\end{align}

\begin{lemma}[Convergence of \(\Phi_\rho\) to \(\Phi_0\)]
\label{pf:lem:Phi-rho-to-Phi0}
We have,
\begin{equation}
\label{pf:eq:Phi-rho-to-Phi0-estimate}
  \|\Phi_\rho-\Phi_0\|_{L^\infty(\Omega_\rho)}
  \le C\rho^{\alpha_R}+C\rho|\ln\rho|
\end{equation}
for some \(\alpha_R\in(0,1]\).  
\end{lemma}

\begin{proof}[Reference]
This is BBH Lemma I.2 and the oscillation lemma used in the shrinking-hole construction \cite[Section I.4]{BBH}, with standard elliptic estimates on Lipschitz/curvilinear polygonal domains. In a smooth-domain situation this is the BBH \(O(\rho)\) estimate.
For a nonsmooth domain, the exponent \(\alpha_R\) comes from H\"older
regularity of \(R_0\) \cite[Proposition~3.6]{Nittka2011Regularity}.
\begin{equation}
  |R_0(x)-R_0(y)|\le C|x-y|^{\alpha_R}.
\end{equation} 
The last term $O(\rho|\ln\rho|)$ accounts for the small flux errors and
the vertex pieces removed from the boundary.
\end{proof}

\begin{proposition}
On $\Omega_{\rho}$, we have 
\begin{align}
  \frac12\int_{\Omega_\rho}|\nabla\Phi_\rho|^2\,\dd x
  ={}&\pi\sum_{i=1}^Nd_i^2\ln\frac1\rho
  +\sum_{j=1}^K\frac{\eta_j^2}{2\omega_j}
  \ln\frac1\rho
  \notag\\
  &+W(\mathcal C)+o(1),
  \label{eq:finite-part-main}
\end{align}
where the finite part $W(\mathcal C)$ is given by
\begin{equation}
\begin{split}
\label{eq:general-W}
  W(\mathcal C)
  &=\frac12 \int_{\partial\Om\backslash V}\Phi_0q\,ds
    -\pi\sum_{i=1}^Nd_iR_0(a_i)
    +\pi\sum_{j=1}^K c_jR_0(b_j) \\
&-\pi \sum_{1\leq i\neq j\leq N} d_i d_j \ln|a_i-a_j|  - \pi \sum_{1\leq i\neq j\leq K} \frac{2\pi}{\omega_j} c_i c_j \ln|b_i-b_j|\\
&+ \pi\sum_{i = 1}^N\sum_{j = 1}^K \left(\frac{2\pi}{\omega_j}+1\right) d_i c_j \ln|b_j-a_i|
\end{split}
\end{equation}
\end{proposition}
\begin{proof}
Since \(\Phi_\rho\) is harmonic in \(\Omega_\rho\),
\begin{align}
\label{pf:eq:Phi-rho-green-identity}
  \int_{\Omega_\rho}|\nabla\Phi_\rho|^2
  &=\int_{\partial \Omega_\rho}\Phi_\rho\partial_\nu\Phi_\rho \\
  &=\int_{\Gamma_\rho}\Phi_\rho q\,ds
    -\sum_{i=1}^N2\pi d_i\Phi_\rho(x_i)
    +\sum_{j=1}^K\eta_j\Phi_\rho(y_j),
\end{align}
where \(x_i\in \Gamma_{vi}(\rho)\) and \(y_j\in \Gamma_{wj}(\rho)\) are arbitrary points.  Using Lemma~\ref{pf:lem:Phi-rho-to-Phi0},
\begin{equation}
\label{pf:eq:Phi0-green-identity}
  \frac12\int_{\Omega_\rho}|\nabla\Phi_\rho|^2
  =\frac12\int_{\Gamma_\rho}\Phi_0q\,ds
  -\pi\sum_{i=1}^N d_i\Phi_0(x_i)
  +\frac12\sum_{j=1}^K\eta_j\Phi_0(y_j)+o(1).
\end{equation}

For \(x_i\in \Gamma_{vi}(\rho)\),
\begin{equation}
\label{pf:eq:Phi0-expansion-inner-value}
  \Phi_0(x_i)=d_i\ln\rho+A_i+o(1),
\end{equation}
where
\begin{equation}
  A_i:=R_0(a_i)+\sum_{k\ne i}d_k\ln|a_i-a_k|-
  \sum_{j=1}^K\frac{\eta_j}{\omega_j}\ln|a_i-b_j|.
\end{equation}
For \(y_j\in \Gamma_{wj}(\rho)\),
\begin{equation}
\label{pf:eq:Phi0-expansion-corner-value}
  \Phi_0(y_j)=-\frac{\eta_j}{\omega_j}\ln\rho+B_j+o(1),
\end{equation}
where
\begin{equation}
  B_j:=R_0(b_j)+\sum_{i=1}^N d_i\ln|b_j-a_i|-
  \sum_{\ell\ne j}\frac{\eta_\ell}{\omega_\ell}\ln|b_j-b_\ell|.
\end{equation}
The boundary integral has a finite-part limit
\begin{equation}\label{pf:eq:Phi0_q_boundary}
  \int_{\partial \Omega\setminus V}\Phi_0q\,ds
  =\lim_{\rho\to0}\int_{\Gamma_\rho}\Phi_0q\,ds.
\end{equation}
Indeed, near a vertex \(\Phi_0=O(|\ln r|)\) and \(q\) is bounded on each open side, so the missing pieces have size \(O(\rho|\ln\rho|)\).
Combining \eqref{pf:eq:Phi0-green-identity}--
\eqref{pf:eq:Phi0_q_boundary} gives the expansion
\eqref{eq:finite-part-main}.  Substitution of the finite values $A_i$ and
$B_j$ above produces the boundary, defect--defect, defect--corner, and
corner--corner terms in \eqref{eq:general-W}.  
\end{proof}

\subsection{Derivative and vanishing gradient}\label{pf:sec:criticality}
Near an interior vortex $a_i$, write
\begin{equation}
  \Phi_0(x)=d_i\ln|x-a_i|+S_i(x),
  \label{eq:local-stream-main}
\end{equation}
where $S_i$ is harmonic.  Let $H_i$ be its harmonic conjugate, with the
orientation fixed by $\nabla H_i=J\nabla S_i$ and its additive constant
chosen so that it is the local regular phase in
\begin{align}
  u(z)&=\alpha_i
  \left(\frac{z-a_i}{|z-a_i|}\right)^{d_i}e^{\ii H_i(z)},
  &&\alpha_i\in\Sone,
  &&\text{near }a_i.
\end{align}

\begin{lemma}[Derivative of the renormalized energy]
\label{prop:derivative-main}
For each interior vortex,
\begin{equation}
  \nabla_{a_i}W(\mathcal C)
  =-2\pi d_i\nabla S_i(a_i)
  =2\pi d_iJ\nabla H_i(a_i).
  \label{eq:derivative-main}
\end{equation}
Thus $\nabla_{a_i}W=0$ if and only if $\nabla H_i(a_i)=0$.
\end{lemma}

\begin{proof}
This is the BBH derivative formula for the renormalized energy, in the ``vanishing gradient'' discussion of Chapter VIII \cite[Chapter VIII.2--VIII.3, especially Theorem VIII.3]{BBH}. The corner defects are fixed during the variation of $a_i$, so they enter only through the regular part $S_i$.
\end{proof}

\begin{proposition}[Vanishing gradient for critical limits]
\label{thm:vanishing-main}
For every interior singularity produced by
Theorem~\ref{thm:clearing-main},
\begin{equation}
  \nabla H_i(a_i)=0,
  \qquad\text{and hence}\qquad
  \nabla_{a_i}W(\mathcal C)=0.
  \label{eq:vanishing-main}
\end{equation}
\end{proposition}
\begin{proof}
Fix an interior singularity $a_i$ and choose $r_0>0$ so that
$\overline{B(a_i,r_0)}\Subset \Omega\setminus(V\cup(A\setminus\{a_i\}))$.
For a solution $U_\eps$ of \eqref{eq:GL-EL}, define the stress tensor
\[
  (T_\eps)_{pq}
  =\partial_pU_\eps\cdot\partial_qU_\eps
  -\left\{
    \frac12|\nabla U_\eps|^2
    +\frac1{4\eps^2}(1-|U_\eps|^2)^2
  \right\}\delta_{pq}.
\]
The Euler--Lagrange equation gives $\divop T_\eps=0$.  Hence, for every
$0<r<r_0$,
\begin{equation}
  \int_{\partial B(a_i,r)}T_\eps\nu\,\dd s=0.
  \label{pf:eq:stress-flux-zero}
\end{equation}

On a fixed annulus about $a_i$, Lemma~\ref{pf:lem:scalar}(iii) gives
$\rho_\eps\to1$ exponentially, while
Theorem~\ref{thm:clearing-main} gives smooth convergence of $v_\eps$ to
$u_*$.  Thus $U_\eps=\rho_\eps v_\eps\to u_*$ in $C^2$ there.  Moreover,
the equation and the fact that $|U_\eps|$ is bounded away from zero imply
that $\eps^{-2}(1-|U_\eps|^2)$ is locally bounded, so the potential part of
$T_\eps$ tends uniformly to zero on the fixed circle.  Passing to the
limit in \eqref{pf:eq:stress-flux-zero} yields
\begin{equation}
  \int_{\partial B(a_i,r)}T_*(u_*)\nu\,\dd s=0,
  \qquad
  T_*(u_*)=\nabla u_*^T\nabla u_*
  -\frac12|\nabla u_*|^2I.
  \label{pf:eq:limit-stress-flux}
\end{equation}

Near $a_i$, use the local representation associated with
\eqref{eq:local-stream-main}:
\[
  u_*(z)=\alpha_i
  \left(\frac{z-a_i}{|z-a_i|}\right)^{d_i}
  e^{\ii H_i(z)}.
\]
Writing its phase as $\phi=d_i\theta+H_i$ and setting
$\tau=J\nu$ on $\partial B(a_i,r)$ gives
\[
  \partial_\nu\phi=\nabla H_i\cdot\nu,
  \qquad
  \partial_\tau\phi=\frac{d_i}{r}+\nabla H_i\cdot\tau,
\]
and therefore
\[
  T_*(u_*)\nu
  =\frac12\left\{(\partial_\nu\phi)^2
  -(\partial_\tau\phi)^2\right\}\nu
  +(\partial_\nu\phi)(\partial_\tau\phi)\tau.
\]
Let $h_i=\nabla H_i(a_i)$.  Since $H_i$ is smooth, substitution of
$\nabla H_i=h_i+O(r)$ into \eqref{pf:eq:limit-stress-flux}, followed by
$r\rightarrow0$, gives
\begin{align*}
  0
  &=d_i\int_0^{2\pi}
  \left[-(h_i\cdot\tau)\nu+(h_i\cdot\nu)\tau\right]\dd\theta\\
  &=2\pi d_iJh_i.
\end{align*}
Here the $d_i^2/r$ term integrates to zero because
$\int_0^{2\pi}\nu\,\dd\theta=0$, and every remaining quadratic regular
term is $O(r)$.  Since $d_i\ne0$, this proves
$\nabla H_i(a_i)=0$.  Proposition~\ref{prop:derivative-main} then gives
$\nabla_{a_i}W=0$.  This is the local stress-energy calculation used in
BBH Theorem~X.5 \cite[Chapters VII--VIII and Theorem X.5]{BBH}; the choice
of an interior ball is why no corner contribution appears.
\end{proof}

\section{Conclusion and outlook}\label{sec:conclusion}

We have developed an asymptotic description of 
Ginzburg--Landau (GL) critical points on two-dimensional domains with corners
and Dirichlet data that vanish at the vertices. The weighted
Lassoued--Mironescu splitting separates the forced scalar boundary layer
from the phase and vortex structure. Under the reduced logarithmic energy
bound and vertex logarithmic tightness, fixed-truncation estimates yield subsequential local smooth
convergence of the normalized maps away from the interior vortices to a
canonical harmonic map, incorporating both integer interior degrees and
fractional corner singularities determined by the corner angles and lifted
phase jumps. Subtracting the interior and corner logarithmic divergences
defines the corresponding renormalized energy. A local stress-energy
argument then shows that the limiting interior vortex configuration is a
critical point of this energy, with the vertices and the discrete
topological data held fixed. Thus the renormalized energy describes the
equilibrium conditions for limiting critical configurations in this
setting, without requiring energy minimality.

Several questions remain for future work. The present result establishes
a necessary condition on limits of admissible GL solutions;
the converse is to determine which critical configurations of the
renormalized energy can be realized by families of such solutions. 
This would also raise the question of how
the stability and Morse index of the GL solutions relate to the
Hessian of the renormalized energy, while degenerate configurations and
higher-degree vortex cores may require additional solvability conditions.
In two dimensions, another direction is to replace the prescribed
Dirichlet trace by weak anchoring, through a boundary penalization and its
associated Robin-type condition, or to consider Neumann and mixed boundary
conditions. One would then need to determine how the boundary condition
and, for weak anchoring, the scaling of the anchoring strength with
$\eps$ affect boundary defects, corner phase selection, and the effective
interaction energy. Extending the analysis to three dimensions presents
a different geometric challenge: for the complex-valued Ginzburg--Landau
model, vortices are typically curves rather than isolated points.
The limiting equilibrium problem therefore involves
variations of entire curves and their interaction with the boundary.
Controlling concentration and multiplicity along these curves, possible
singular limits, and their behaviour near edges and vertices of
polyhedral domains would require tools beyond the planar harmonic-map
construction.

\bibliographystyle{plain}
\bibliography{main}

@article{ASnonvanishing,
  author  = {Andr{\'e}, Nelly and Shafrir, Itai},
  title   = {Minimization of a {Ginzburg--Landau} type functional with nonvanishing {Dirichlet} boundary condition},
  journal = {Calculus of Variations and Partial Differential Equations},
  volume  = {7},
  pages   = {191--217},
  year    = {1998}
}

@article{ASzeros,
  author  = {Andr{\'e}, Nelly and Shafrir, Itai},
  title   = {On the minimizers of a {Ginzburg--Landau}-type energy when the boundary condition has zeros},
  journal = {Advances in Differential Equations},
  volume  = {9},
  number  = {7--8},
  pages   = {891--960},
  year    = {2004}
}

@article{BerestyckiBrezis1980FreeBoundaryPlasma,
  author  = {Berestycki, Henri and Brezis, Ha{\^i}m},
  title   = {On a free boundary problem arising in plasma physics},
  journal = {Nonlinear Analysis: Theory, Methods \& Applications},
  volume  = {4},
  number  = {3},
  pages   = {415--436},
  year    = {1980},
  doi     = {10.1016/0362-546X(80)90083-8}
}

@book{BBH,
  author    = {Bethuel, Fabrice and Brezis, Ha{\"\i}m and H{\'e}lein, Fr{\'e}d{\'e}ric},
  title     = {{Ginzburg--Landau} Vortices},
  series    = {Progress in Nonlinear Differential Equations and Their Applications},
  volume    = {13},
  publisher = {Birkh{\"a}user},
  address   = {Boston},
  year      = {1994}
}

@article{BOS,
  author  = {Bethuel, Fabrice and Orlandi, Giandomenico and Smets, Didier},
  title   = {Improved estimates for the {Ginzburg--Landau} equation: the elliptic case},
  journal = {Annali della Scuola Normale Superiore di Pisa, Classe di Scienze},
  series  = {5},
  volume  = {4},
  pages   = {319--355},
  year    = {2005}
}

@article{BO,
  author  = {Brezis, Ha{\"\i}m and Oswald, Luc},
  title   = {Remarks on sublinear elliptic equations},
  journal = {Nonlinear Analysis: Theory, Methods \& Applications},
  volume  = {10},
  number  = {1},
  pages   = {55--64},
  year    = {1986}
}

@book{Evans,
  author    = {Evans, Lawrence C.},
  title     = {Partial Differential Equations},
  series    = {Graduate Studies in Mathematics},
  volume    = {19},
  edition   = {2},
  publisher = {American Mathematical Society},
  address   = {Providence, RI},
  year      = {2010}
}

@article{Kato1972,
  author  = {Kato, Tosio},
  title   = {Schr{\"o}dinger operators with singular potentials},
  journal = {Israel Journal of Mathematics},
  volume  = {13},
  pages   = {135--148},
  year    = {1972}
}

@article{LM,
  author  = {Lassoued, Lotfi and Mironescu, Petru},
  title   = {{Ginzburg--Landau} type energy with discontinuous constraint},
  journal = {Journal d'Analyse Math{\'e}matique},
  volume  = {77},
  pages   = {1--26},
  year    = {1999}
}

@article{lin1995solutions,
  author  = {Lin, Fang-Hua},
  title   = {Solutions of {Ginzburg--Landau} equations and critical points of the renormalized energy},
  journal = {Annales de l'Institut Henri Poincar{\'e} C, Analyse Non Lin{\'e}aire},
  volume  = {12},
  number  = {5},
  pages   = {599--622},
  year    = {1995}
}

@book{Mattila1999,
  author    = {Mattila, Pertti},
  title     = {Geometry of Sets and Measures in Euclidean Spaces: Fractals and Rectifiability},
  series    = {Cambridge Studies in Advanced Mathematics},
  volume    = {44},
  publisher = {Cambridge University Press},
  address   = {Cambridge},
  year      = {1999}
}

@article{Nittka2011Regularity,
  author  = {Nittka, Robin},
  title   = {Regularity of solutions of linear second order elliptic and parabolic boundary value problems on {Lipschitz} domains},
  journal = {Journal of Differential Equations},
  volume  = {251},
  number  = {4--5},
  pages   = {860--880},
  year    = {2011},
  doi     = {10.1016/j.jde.2011.05.019}
}

\end{document}